\documentclass[a4paper,11pt]{amsart}

\scrollmode
\usepackage{amsmath,amssymb,amsthm,amsfonts,graphicx,color,fancyhdr}
\usepackage[latin1]{inputenc}
\usepackage[colorlinks=true, pdfstartview=FitV, linkcolor=blue, citecolor=blue, urlcolor=blue]{hyperref}

 \usepackage{color}
 \definecolor{DarkBlue}{rgb}{0.0,0.0,0.8} 
\definecolor{DarkGreen}{rgb}{0.0,0.55,0.14}
\definecolor{DarkRed}{rgb}{0.5,0,0.07}

\newtheorem{thm}{Theorem}[section]
\newtheorem{lem}[thm]{Lemma}
\newtheorem{prop}[thm]{Proposition}
\newtheorem{cor}[thm]{Corollary}

\theoremstyle{definition}
\newtheorem{dfn}[thm]{Definition}

\newtheorem{rem}[thm]{Remark}
\newtheorem{ass}[thm]{Assumption}
\newtheorem{ex}[thm]{Example}

\def\be#1 {\begin{equation} \label{#1}}
\newcommand{\ee}{\end{equation}}

\newcommand{\mb}{\medskip\noindent}

\newcommand{\R}{\mathbb R}

\newcommand{\C}{\mathbb C}

\newcommand{\A}{{\mathcal A}}
\newcommand{\K}{\kappa}

\def\Xint#1{\mathchoice
   {\XXint\displaystyle\textstyle{#1}}%
   {\XXint\textstyle\scriptstyle{#1}}%
   {\XXint\scriptstyle\scriptscriptstyle{#1}}%
   {\XXint\scriptscriptstyle\scriptscriptstyle{#1}}%
   \!\int}
\def\XXint#1#2#3{{\setbox0=\hbox{$#1{#2#3}{\int}$}
     \vcenter{\hbox{$#2#3$}}\kern-.5\wd0}}

\def\aver#1{\Xint-_{#1}}

\newcommand{\M}{{\mathcal M}}
\author{Nadine Badr}
\address{Nadine Badr - Universit\'e de Lyon, CNRS, Universit\'e Lyon 1 \\ Institut Camille Jordan \\ 43 boulevard du 11 Novembre 1918 \\ F-69622 Villeurbanne Cedex, France}
\email{badr@math.univ-lyon1.fr}
\urladdr{http://math.univ-lyon1.fr/~badr/}

\author{Fr\'ed\'eric Bernicot}
\address{Fr\'ed\'eric Bernicot - CNRS - Universit\'e Lille 1 \\ Laboratoire de math\'ematiques Paul Painlev\'e \\ 59655 Villeneuve d'Ascq Cedex, France}
\email{frederic.bernicot@math.univ-lille1.fr}
\urladdr{http://math.univ-lille1.fr/~bernicot/}

\author{Emmanuel Russ}
\address{Emmanuel Russ - Institut Fourier \\ Universit\'e Joseph Fourier Grenoble I \\ 100 rue des maths, BP 74 \\ 38402 Saint-Martin-d'H\`eres Cedex, France}
\email{eruss@fourier.ujf-grenoble.fr}
\urladdr{http://www.cmi.univ-mrs.fr/~russ/}

\title[Algebra properties for Sobolev spaces]{Algebra properties for Sobolev spaces- Applications to semilinear PDE's on manifolds}

\date{July 19, 2011}

\begin{document}



\begin{abstract} In this work, we aim to prove algebra properties for generalized Sobolev spaces $W^{s,p} \cap L^\infty$ on a Riemannian manifold, where $W^{s,p}$ is of Bessel-type $W^{s,p}:=(1+L)^{-s/m}(L^p)$ with an operator $L$ generating a heat semigroup satisfying off-diagonal decays. We don't require any assumption on the gradient of the semigroup. To do that, we propose two different approaches (one by a new kind of paraproducts and another one using functionals). We also give a chain rule and study the action of nonlinearities on these spaces and give applications to semi-linear PDEs. These results are new on Riemannian manifolds (with a non bounded geometry) and even in the Euclidean space for Sobolev spaces associated to second order uniformly elliptic operators in divergence form.
\end{abstract}

\maketitle

\begin{quote}
\footnotesize\tableofcontents
\end{quote}

\section{Introduction}

\subsection{The Euclidean setting}

 It is known that in  $\mathbb{R}^d$,  the Bessel potential space 
$$W^{\alpha,p}_{\Delta}=\left\lbrace f \in L^p; \, \Delta^{\alpha/2}f \in L^p  \right\rbrace,
$$
 is an algebra under the pointwise product for all $1< p<\infty$ and $\alpha >0$ such that $\alpha p>d$. This result is due to Strichartz \cite{St}.

Twenty years after Strichartz work, Kato and Ponce \cite{KP} gave a stronger result. They proved that for all $1< p<\infty$ and  $\alpha >0$,
$W^{\alpha,p}_{\Delta}\cap L^{\infty}$ is an algebra under the pointwise product. 
Nowadays, these properties and more general Leibniz rules can be ``easily" obtained in the Euclidean setting using paraproducts and boundedness of these bilinear operators. This powerful tool allows us to split the pointwise product in three terms, the regularity of which can be easily computed with the ones of the two functions.

There is also a work of Gulisashvili and Kon \cite{GK} where it is shown that the algebra property remains true considering the homogeneous Sobolev spaces. That is for all  $1< p<\infty$ and  $\alpha >0$,
$\dot W^{\alpha,p}_{\Delta}\cap L^{\infty}$ is an algebra under the pointwise product. 

The main motivation of such inequalities (Leibniz rules and algebra properties) comes from the study of nonlinear PDEs. In particular, to obtain well-posedness results in Sobolev spaces for some semi-linear PDEs, we have to understand how the nonlinearity acts on Sobolev spaces. This topic, the action of a nonlinearity on Sobolev spaces (and more generally on Besov spaces), has given rise to numerous works in the Euclidean setting where the authors try to obtain the minimal regularity on a nonlinearity $F$ such that the following property holds
$$ f\in B^{s,p} \Longrightarrow F(f) \in B^{s,p}$$
where $B^{s,p}$ can be (mainly) Sobolev spaces or more generally Besov spaces (see for example the works of Sickel \cite{Sickel} with Runst \cite{RS} and the work of Bourdaud \cite{Bourdaud} ... we refer the reader to \cite{BS} for a recent survey on this topic).

\subsection{On Riemannian manifolds}

Analogous problems in a non Euclidean context do not seem to have been considered very much. In \cite{CRT}, Coulhon, Russ and Tardivel extended these results of algebra property to the case of Lie groups and also for Riemannian manifolds with Ricci curvature bounded from below and positive injective radius. In this setting, the heat kernel and its gradient satisfy pointwise Gaussian upper bounds, which play a crucial role in the proof.
\bigskip

The goal of this paper is to study the algebra property of Sobolev spaces on more general Riemannian manifolds. More precisely, we want to extend the above result to the case of a  Riemannian manifold with weaker geometric hypotheses and also to more general Sobolev spaces. Namely, we will consider a general semigroup on a Riemannian manifold with off-diagonal decays and obtain results under $L^p$ boundedness of the Riesz transform, which allows us to weaken the assumptions in \cite{CRT}. \par
\noindent In particular, we never use pointwise estimates for the kernel of the semigroup or its gradient. Recall that a pointwise Gaussian upper bound for the gradient of the heat kernel on a complete Riemannian manifold $M$ is known for instance when $M$ has non-negative Ricci curvature, which is a rather strong assumption. More generally, this poinwise Gaussian upper bound can be characterized in terms of Faber-Krahn inequalities for the Laplace Beltrami operator under Dirichlet boundary condition (see \cite{grigo}). Instead of that, all our proofs only rely on off-diagonal decays for the semigroup (see Assumption \ref{ass} below), which are satisfied, for instance, by the heat semigroup under weaker assumptions (see Example \ref{exlaplace} below). \par
\noindent Recall that there are several situations in which one encounters operators which satisfy off-diagonal decays even though their kernels do not satisfy pointwise estimates. In spite of this lack of pointwise estimates, it is still possible to develop a lot of analysis on these operators (see for instance \cite{hofmay,hmm} in the Euclidean case, \cite{amr} in the case of Riemannian manifolds). The present work shows that these off-diagonal decays in conjunction with the $L^p$ boundedness of Riesz transforms for $p$ close to $2$ are enough to yield the Leibniz rule for Sobolev spaces associated with the operators under consideration.\par
\noindent Aiming that, we will propose two different approaches. On the one hand, we will extend the method introduced in \cite{CRT} using characterization of Sobolev spaces with square functionals. On the other hand, we will extend the ``paraproducts point of view" to this framework. Classical paraproducts are defined via Fourier transforms, however in \cite{B2} the first author has introduced analogues of such bilinear operators relatively to a semigroup (see \cite{Phd} for another independent work of Frey). We will also describe how to use them to prove Leibniz rules and algebra properties for Sobolev spaces in a general context.

\subsection{Results}

By general Sobolev spaces, we mean the following:
\begin{dfn} Let $L$ be a linear operator of type $w\in [0,\pi/2)$ satisfying assumption (\ref{ass}) below (see Section 2), which can be thought of as an operator of order  $m>0$. For $1\leq p<\infty$ and $s>0$, we define the homogeneous Sobolev spaces $\dot W_L^{s,p}$ as
$$ \dot W_L^{s,p}:=\left\{ f\in L^{p}_{loc},\ L^{s/m}(f)\in L^p\right\}$$
with the semi-norm
$$\|f\|_{\dot W_L^{s,p}}:=\|L^{s/m}(f)\|_{L^p}.$$
And we define the non-homogeneous Sobolev spaces $ W_L^{s,p}$ as
$$  W_L^{s,p}:=\left\{ f\in L^p,\ L^{s/m}(f)\in L^p\right\}$$ 
with the norm
$$\|f\|_{ W_L^{s,p}}:=\|f\|_{L^p}+\|L^{s/m}(f)\|_{L^p}.$$
\end{dfn}

Concerning the algebra property of these Sobolev spaces, we first obtain the following generalized Leibniz rule:

\begin{thm}\label{Sobalg1}Let $M$ be a complete Riemannian manifold satisfying the doubling property $(D)$  and a local Poincar\'e inequality $(P_{s,loc})$ for some $1\leq s<2$. Let $L$ be a linear operator of type $w\in [0,\pi/2)$ satisfying Assumption (\ref{ass}) below (see section 2). Let $\alpha\in [0,1)$ and $r>1$ with $s\leq r<\infty$. Let $p_1,\, q_2\in[r,\infty)$, $q_1,p_2\in(r, \infty]$ verifying 
$$\frac{1}{r}=\frac{1}{p_i}+\frac{1}{q_i}$$ and
$r,p_1,q_2\in(s_-,s_+)$, $q_1,p_2\in(s_-,\infty]$ (see Assumption (\ref{ass}) below for the definition of $s_-,s_+$). \\
Then for all $f\in W^{\alpha,p_1}_L\cap L^{p_2}$ and $g\in W^{\alpha,q_2}_L\cap L^{q_1}$, we have $fg\in W^{\alpha,r}_L$ with 
$$ \|fg\|_{W_L^{\alpha,r}} \lesssim \|f\|_{W_L^{\alpha,p_1}} \|g\|_{L^{q_1}} + \|f\|_{L^{p_2}} \|g\|_{W_L^{\alpha,q_2}}.$$
Moreover, if we assume a global Poincar\'e inequality $(P_s)$, we have
$$ \|L^{\frac{\alpha}{m}}(fg)\|_{L^r} \lesssim \|L^{\frac{\alpha}{m}}f\|_{L^{p_1}} \|g\|_{L^{q_1}} + \|f\|_{L^{p_2}} \|L^{\frac{\alpha}{m}}g\|_{L^{q_2}}.$$
\end{thm}

\mb As a consequence (with $q_1=p_2=\infty$ and $r=p=p_1=q_2$), we get the algebra property for $W^{\alpha,p}_L\cap L^{\infty}$, more precisely:

\begin{thm}\label{Sobalg2}Let $M$ be a complete Riemannian manifold satisfying the doubling property $(D)$  and a local Poincar\'e inequality $(P_{s,loc})$  for some $1\leq s<2$. Let $L$ a linear operator of type $w\in [0,\pi/2)$ satisfying Assumption (\ref{ass}). Let $\alpha\in [0,1)$ and  $p\in(\max(s,s_-),s_+)$. Then the space $W^{\alpha,p}_L\cap L^{\infty}$ is an algebra under the pointwise product. More precisely, for all $f,g\, \in W^{\alpha,p}_L\cap L^{\infty}$, one has $fg \in W^{\alpha,p}_L\cap L^{\infty}$  with
$$ \|fg\|_{W_L^{\alpha,p}} \lesssim \|f\|_{W_L^{\alpha,p}} \|g\|_{L^{\infty}} + \|f\|_{L^\infty} \|g\|_{W_L^{\alpha,p}}.$$
Moreover, if we assume a global Poincar\'e inequality $(P_s)$, the homogeneous Sobolev space  $\dot W^{\alpha,p}_L\cap L^{\infty}$ is an algebra under the pointwise product. More precisely, for all $f,\,g\, \in \dot W^{\alpha,p}_L\cap L^{\infty}$, then $fg \in \dot W^{\alpha,p}_L\cap L^{\infty}$  with
$$ \|L^{\alpha/m}(fg)\|_{L^p} \lesssim \|L^{\alpha/m}f\|_{L^{p}} \|g\|_{L^{\infty}} + \|f\|_{L^{\infty}} \|L^{\alpha/m}g\|_{L^{p}}.$$
\end{thm}

Consequently, we will deduce the following algebra property for Sobolev spaces:
 
\begin{thm}\label{Sobalg3}Let $M$ be a complete Riemannian manifold satisfying the doubling property $(D)$ and a local Poincar\'e inequality $(P_{s,loc})$ for some $1\leq s<2$. Let $L$ be a linear operator of type $w\in [0,\pi/2)$ satisfying Assumption (\ref{ass}).
Moreover, we assume that $M$ satisfies the  following  lower bound of the volume of small balls
\begin{equation}\tag{$MV_d$}
 \mu(B(x,r)) \gtrsim r^d,
\end{equation}
for all $0<r\leq 1$.
  Let $\alpha\in [0,1)$ and  $p\in(\max(s,s_-),s_+)$ such that $\alpha p>d$ where $d$ is the homogeneous dimension. Then the space $W^{\alpha,p}_L$ is included in $L^\infty$ and is an algebra under the pointwise product. More precisely, for all $f,\,g\, \in W^{\alpha,p}_L$, one has $fg \in W^{\alpha,p}_L$  with
$$ \|fg\|_{W_L^{\alpha,p}} \lesssim \|f\|_{W_L^{\alpha,p}} \|g\|_{W_L^{\alpha,p}}.$$
\end{thm}

These two theorems are a particular case of Theorem \ref{Sobalg1}. Nevertheless, we will prove them using another method than that of the proof of Theorem \ref{Sobalg1}. Note also that the proof of the three Theorems is trivial when $\alpha=0$. So we will prove them for $\alpha>0$.

\bigskip

To finish, we also consider the case when $\alpha =1$:
\begin{thm}\label{Sobalg4} Let $M$ be a complete Riemannian manifold satisfying the doubling property $(D)$. Assume that the Riesz transform $\nabla L^{-1/m}$ is bounded on $ L^p$ for $p\in (s_-,s_+)$ and that we have the reverse Riesz inequalities
$$
\|L^{1/m}f\|_{L^p}\lesssim \|\nabla f\|_{L^p}$$
for all   $p\in (q_-,q_+)$.
Let  $1\leq q_-<r<q_+$. Let $p_1\in(s_-,s_+)$ with $p_1\geq r$, $q_1\in(r, \infty]$ and $q_2\in (s_-,s_+)$ with $q_2\geq r$, $p_2\in(r, \infty]$ verifying 
$$\frac{1}{r}=\frac{1}{p_i}+\frac{1}{q_i}.$$
Then for all $f\in W^{1,q_1}_L\cap L^{p_1}$, $g\in W^{1,q_2}_L\cap L^{p_2}$, $fg\in W^{1,r}_L$ with 
$$ \|fg\|_{W_L^{1,r}} \lesssim \|f\|_{W_L^{1,p_1}} \|g\|_{L^{q_1}} + \|f\|_{L^{p_2}} \|g\|_{W_L^{1,q_2}}.$$
Consequenly, for $p\in (\max(q_-,s_-),\min(q_+,s_+))$,  $W^{1,p}_L\cap L^{\infty}$ is an algebra under the pointwise product. If moreover, $p>d$ and $M$ satisfies $(MV_d)$, then the Sobolev space $W^{1,p}_L$ is also an algebra under the pointwise product.
\end{thm}

These results are new comparing with \cite{CRT} for Riemannian manifolds with bounded geometry (see Example \ref{exlaplace}) and even in the Euclidean case when $L$ is for example an elliptic operator, appearing in Kato conjecture and whose functional spaces were introduced in \cite{hmm} (see Example \ref{exam}).

\begin{rem}
1. Assuming only the boundedness of the local Riesz tranform and its reverse inequalities, the non-homogeneous result  of Theorem \ref{Sobalg4} still holds. 
\\
2. Taking the usual Sobolev space defined by the gradient, when $\alpha=1$, Theorem \ref{Sobalg4} holds without any restriction on the exponents.  It suffices to use the Leibniz rule and  H\"older inequality. 
The assumptions on the Riesz transform and the reverse inequalities reduce the proof of Theorem \ref{Sobalg4}  to the usual Leibniz rules.
\end{rem}

\begin{proof}[Proof of Theorem \ref{Sobalg4}]
We have
\begin{align*}
\|L^{1/m}(fg)\|_{L^r}&\leq \|\nabla (fg)\|_{L^r}
\\
&\leq   \|\nabla f\,g\|_{L^r}+\|f\,\nabla g\|_{L^r}
\\
&\leq \|\nabla f\|_{L^{p_1}} \|g\|_{L^{q_1}}+\|f\|_{L^{p_2}}\|\nabla g\|_{L^{q_2}} 
\\
&
\leq \|L^{1/m} f\|_{L^{p_1}} \|g\|_{L^{q_1}}+\|f\|_{L^{p_2}}\|L^{1/m} g\|_{L^{q_2}}
\end{align*}
where in the first inequality, we used the $L^r$  reverse Riesz inequality  and in the last inequality we used the boundedness of the Riesz transform on $L^p$ for $p=p_1$ and $p=q_2$.
\end{proof}

Unfortunately, we are not able to have such a positive result when $\alpha>1$.   In this case, we need the boundedness of the iterated Riesz transforms which has not been studied until now on a general Riemannian manifold,  even when $L$ is the Laplace-Beltrami operator. We also describe results for such higher order Sobolev spaces in the context of sub-Riemannian structure (where a chain rule holds).

\bigskip

The plan of the paper is a follows. In section  \ref{sec:pre}, we recall the definitions of the hypotheses that we assume on our manifold and the linear operator $L$. We prove Sobolev embeddings for the generalized Sobolev spaces in section 3. Using a new point of view: the paraproducts, we prove Theorem \ref{Sobalg1} in section 4. Section 5 is devoted to the proof of Theorems \ref{Sobalg2} and \ref{Sobalg3} characterizing the Sobolev spaces using a representation formula in terms of first order differences. In Section 6, we will briefly describe extension to higher order Sobolev spaces under a sub-Riemannian structure and we will study how nonlinearities act on the Sobolev spaces. Finally, in section 7, we give applications of our result in PDE. We obtain well posedness result  for Schr\"odinger equations and also for heat equations associated to the operator $L$.

\section{Preliminaries} \label{sec:pre}

 For a ball $B$ in a metric space, $\lambda B$  denotes the ball co-centered with $B$ and with radius $\lambda$ times that of $B$. Finally, $C$ will be a constant that may change from an inequality to another and we will use $u\lesssim
v$ to say that there exists a constant $C$  such that $u\leq Cv$ and $u\simeq v$ to say that $u\lesssim v$ and $v\lesssim u$.

\mb In all this paper, $M$ denotes a complete Riemannian manifold. We write $\mu$ for the Riemannian measure on $M$, $\nabla$ for the
Riemannian gradient, $|\cdot|$ for the length on the tangent space (forgetting the subscript $x$ for simplicity) and
$\|\cdot\|_{L^p}$ for the norm on $ L^p:=L^{p}(M,\mu)$, $1 \leq p\leq +\infty.$  We denote by $B(x, r)$ the open ball of
center $x\in M $ and radius $r>0$.
We deal with the Sobolev spaces of order $1$, $W^{1,p}:=W^{1,p}(M)$, where the norm is defined by:
$$ \| f\|_{W^{1,p}(M)} : = \| f\|_{L^p}+\|\, |\nabla f|\,\|_{L^p}.$$
We write ${\mathcal S}(M)$ for the Schwartz space on the manifold $M$ and ${\mathcal S}'(M)$ for its dual, corresponding to the set of distributions. Moreover in all this work, ${\bf 1}={\bf 1}_M$ will be used for the constant function, equals to one on the whole manifold.

\subsection{The doubling property}
\begin{dfn}[Doubling property] Let $M$ be a Riemannian manifold. One says that $M$ satisfies the doubling property $(D)$ if there exists a constant $C_0>0$, such that for all $x\in M,\, r>0 $ we have
\begin{equation*}\tag{$D$}
\mu(B(x,2r))\leq C_0 \mu(B(x,r)).
\end{equation*}
\end{dfn}

\begin{lem} Let $M$ be a Riemannian manifold satisfying $(D)$ and let $d:=log_{2}C_0$. Then for all $x,\,y\in M$ and $\theta\geq 1$
\begin{equation}\label{eq:d}
\mu(B(x,\theta R))\leq C\theta^{d}\mu(B(x,R)).
\end{equation}
There also exists $c$ and $N\geq 0$, so that for all $x,y\in M$ and $r>0$
\be{eq:N} \mu(B(y,r)) \leq c\left(1+\frac{d(x,y)}{r} \right)^N \mu(B(x,r)). \ee
\end{lem} 
\noindent For example, if $M$ is the Euclidean space $M=\R^d$ then $N=0$ and $c=1$. \\
Observe that if $M$ satisfies $(D)$ then
$$ \textrm{diam}(M)<\infty\Leftrightarrow\,\mu(M)<\infty\,\textrm{ (see \cite{ambrosio1})}. $$
Therefore if $M$ is a non-compact complete Riemannian manifold satisfying $(D)$ then $\mu(M)=\infty$.

\begin{thm}[Maximal theorem]\label{MIT} (\cite{coifman2})
Let $M$ be a Riemannian manifold satisfying $(D)$. Denote by $\M$ the uncentered Hardy-Littlewood maximal function
over open balls of $M$ defined by
 $$ \M f(x):=\underset{\genfrac{}{}{0pt}{}{B \ \textrm{ball}}{x\in B}} {\sup} \ \frac{1}{\mu(B)}\int_{B} |f| d\mu.  $$
Then for every  $p\in(1,\infty]$, $\M$ is $L^p$-bounded and moreover of weak type $(1,1)$\footnote{ An operator $T$ is of weak type $(p,p)$ if there is $C>0$ such that for any $\alpha>0$, $\mu(\{x;\,|Tf(x)|>\alpha\})\leq \frac{C}{\alpha^p}\|f\|_p^p$.}.
\\
Consequently for $s\in(0,\infty)$, the operator $\M_{s}$ defined by
$$ \M_{s}f(x):=\left[\M(|f|^s)(x) \right]^{1/s} $$
is of weak type $(s,s)$ and $L^p$ bounded for all $p\in(s,\infty]$.
\end{thm}

\subsection{Poincar\'e inequality}
\begin{dfn}[Poincar\'{e} inequality on $M$] \label{classP} We say that a complete Riemannian manifold $M$ admits a local Poincar\'{e} inequality $(P_{q})$ for some $q\in[1,\infty)$ if there exists a constant $C>0$ such that, for every function $f\in W^{1,q}_{loc}(M)$ (the set of compactly supported Lipschitz functions on $M$) and every ball $B$ of $M$ of radius $0<r\leq 1$, we have
\begin{equation*}\tag{$P_{qloc}$}
\left(\aver{B}\left|f- \aver{B}f d\mu \right|^{q} d\mu\right)^{1/q} \leq C r \left(\aver{B}|\nabla f|^{q}d\mu\right)^{1/q}.
\end{equation*}
And we say that $M$ admits a global Poincar\'e inequality $(P_q)$ if this inequality holds for all balls $B$ of $M$.
\end{dfn}

Let us recall some known facts about Poincar\'{e} inequalities with varying $q$.
 \\
It is known that $(P_{q})$ implies $(P_{p})$ when $p\geq q$ (see \cite{hajlasz4}). Thus, if the set of $q$ such that
$(P_{q})$ holds is not empty, then it is an interval unbounded on the right. A recent result of S. Keith and X. Zhong
(see \cite{KZ}) asserts that this interval is open in $[1,+\infty[$~:

\begin{thm} \label{thm:kz} Let $(M,d,\mu)$ be a doubling and complete Riemannian manifold, admitting a Poincar\'{e} inequality $(P_{q})$, for  some $1< q<\infty$.
Then there exists $\epsilon >0$ such that $(M,d,\mu)$ admits
$(P_{p})$ for every $p>q-\epsilon$.
\end{thm}

\subsection{Framework for semigroup of operators} \label{subsec:semigroup}

Let us recall the framework of \cite{DY1, DY}. \\
Let $\omega \in[0,\pi/2)$. We define the closed sector in the complex plane ${\mathbb C}$ by
$$ S_\omega:= \{z\in\C,\ |\textrm{arg}(z)|\leq \omega\} \cup\{0\}$$
and denote the interior of $S_\omega$ by $S_\omega^0$.
We set $H_\infty(S^0_\omega)$ for the set of bounded holomorphic functions $b$ on $S_\omega^0$, equipped with the norm
$$ \|b\|_{H_\infty(S_\omega^0)} := \|b\|_{L^\infty(S_\omega^0)}.$$
Then consider a linear operator $L$. It is said of type $\omega$ if its spectrum $\sigma(L)\subset S_\omega$ and for each $\nu>\omega$, there exists a constant $c_\nu$ such that
$$ \left\|(L-\lambda)^{-1} \right\|_{L^2\to L^2} \leq c_\nu |\lambda|^{-1}$$
for all $\lambda\notin S_\nu$.

\mb We refer the reader to \cite{DY1} and \cite{Mc} for more details concerning holomorphic calculus of such operators. In particular, it is well-known that $L$ generates a holomorphic semigroup $(\A_z:=e^{-zL})_{z\in S_{\pi/2-\omega}}$. Let us detail now some assumptions, we make on the semigroup.

\begin{ass} \label{ass}Assume the following conditions: there exist a positive real $m>1$,  exponents $s_-<2<s_+$ and $\delta>1$ with 
\begin{itemize}
 \item For every $z\in S_{\pi/2-\omega}$, the linear operator $\A_z:=e^{-zL}$ satisfies $L^{s_-}-L^\infty$ off-diagonal decay: for all $z$ and ball $B$ of radius $|z|^{1/m}$
 \be{eq:off-diag} \left\|\A_z(f)\right\|_{L^\infty(B)} \lesssim \sum_{k\geq 0} 2^{-\delta k} \left(\aver{2^k B} |f|^{s_-} d\mu \right)^{1/s_-}. \ee
 \item The operator $L$ has a bounded $H_\infty$-calculus on $L^2$. That is, there exists $c_\nu$ such that for $b\in H_\infty(S^0_\nu)$, we can define $b(L)$ as a $L^2$-bounded linear operator and
\be{eq:holocal} \|b(L)\|_{L^2\to L^2} \leq c_\nu \|b\|_{L^\infty}. \ee 
\item The Riesz transform ${\mathcal R}:=\nabla L^{-1/m}$ is bounded on $L^p$ for every $p\in(s_-,s_+)$.
\item For every $t>0$, $e^{-tL}({\bf 1})={\bf 1}$ or equivalently $L({\bf 1})=0$.
\end{itemize}
\end{ass}

\begin{rem} \label{rem:holo} The assumed bounded $H_\infty$-calculus on $L^2$ allows us to deduce some extra properties (see \cite{DY} and \cite{Mc})~:
\begin{itemize}
 \item Due to the Cauchy formula for complex differentiation, pointwise estimate (\ref{eq:off-diag}) still holds for the differentiated semigroup $(tL)^k e^{-tL}$ for every $k\in {\mathbb N}$.
 \item For any holomorphic function $\psi \in H_\infty(S_\nu^0)$ such that for some $s>0$ and for all $z\in S_\nu^0$, $ |\psi(z)|\lesssim \frac{|z|^s}{1+|z|^{2s}},$
the quadratic functional 
\begin{equation} \label{eq:func1} f \rightarrow \left( \int_0^\infty \left|\psi(tL) f \right|^2 \frac{dt}{t} \right)^{1/2} \end{equation}
is $L^2$-bounded.
\item In addition, the Riesz transform is supposed to be bounded in $L^2$ so the following quadratic functionals
\begin{equation} \label{eq:func2} f \rightarrow \left( \int_0^\infty \left| t^{1/m} \nabla \phi(tL) f \right|^2 \frac{dt}{t} \right)^{1/2} \end{equation}
are $L^2$-bounded for any holomorphic function $\phi \in H_{\infty}(S_\nu^0)$ such that for some $s>0$, $ |\phi(z)|\lesssim (1+|z|)^{-2s}.$
\end{itemize}
\end{rem}

\begin{ex} \label{exlaplace} In the case of a doubling Riemannian manifold satisfying Poincar\'e inequality $(P_2)$ and with  $L=-\Delta$ the non-negative Laplacian, then it is well-known (\cite{grigo,saloff-coste}) that heat kernel satisfies pointwise estimates and Assumption (\ref{ass}) also holds with $s_-=1$ (\cite{CD}) and $s_+>2$ (\cite{AC}).
\end{ex}

\begin{ex} \label{exam} Consider a homogeneous elliptic operator $L$ of order $m=2k$ in $\R^d$ defined by
$$ L(f) = (-1)^k \sum_{|\alpha|=|\beta|=k} \partial^\alpha (a_{\alpha\beta} \partial^\beta f),$$
with bounded complex coefficients $a_{\alpha\beta}$.
\begin{itemize}
\item If the coefficients are real-valued  then Gaussian estimates for the heat semigroup hold and Assumption (\ref{ass}) is also satisfied (see Theorem 4 in \cite{Asterisque}).

\item If the coefficients are complex and $d\leq 2k=m$ then the heat kernel satisfies pointwise estimates and so assumption (\ref{ass}) is satisfied for some exponents $s_-,s_+$. We refer the reader to Section 7.2 in \cite{A} for more details. We just point out that, using the interpolation of domains of powers of $L$, the $L^p$ boundedness of $\nabla L^{-1/m}$ for $p\in(s_-,s_+)$ is implied by the $L^p$ boundedness of $\nabla^k L^{-1/2}$ for $p\in (q_-,q_+)$ with $s_-=q_-/m$ and $s_+=(1-1/m)+q_+/m$.

\item Moreover, if the matrix-valued map $A$ is H\"older continuous, then the heat kernel and its gradient admit Gaussian pointwise estimates and so Assumption (\ref{ass}) is satisfied with $s_-=1$ (see \cite{AMT,AMT2}).
\end{itemize}
\end{ex}

\begin{prop} \label{prop:Lp} Under the above assumptions, and since the Riesz transform ${\mathcal R}:=\nabla L^{-1/m}$ is $L^p$ bounded for $p\in (s_-,s_+)$, then the square functionals in (\ref{eq:func1}) and (\ref{eq:func2}) are also bounded in $L^p$ for all $p\in(s_-,s_+)$. Moreover the functionals in (\ref{eq:func1}) are bounded in $L^p$ for every $p\in(s_-,\infty)$.
 \end{prop}

\begin{proof} Let $T$ be one of the square functions in (\ref{eq:func1}). We already know that it is $L^2$ bounded, by holomorphic functional calculus.
Then consider the ``oscillation operator'' at the scale $t$: 
$$B_t:=1-{\mathcal A}_t=1-e^{-tL}=-\int_0^tL e^{-sL} ds. $$
Then, by using differentiation of the semigroup, it is classical that $T B_t$ satisfies $L^2-L^2$ off-diagonal decay at the scale $t^{1/m}$, since the semigroup $e^{-tL}$ is bounded by Hardy Littlewood maximal function ${\mathcal M}_{s_-}$. So we can apply interpolation theory (see \cite{BZ} for a very general exposition of such arguments) and prove that $T$ is bounded on $L^p$ for every $p\in(s_-,2]$ (and then for $p\in[2,\infty)$ by applying a similar reasoning with the dual operators). \\
Then consider a square function $U$ of type (\ref{eq:func2}). Then by using the Riesz transform, it yields
$$ U(f) =\left( \int_0^\infty \left| {\mathcal R} \psi(tL) f \right|^2 \frac{dt}{t} \right)^{1/2}$$
with $\psi(z)=z^{1/m}\phi(z)$. Since ${\mathcal R}$ is supposed to be $L^p$-bounded, it verifies $\ell^2$-valued inequalities and so the $L^p$-boundedness of $U$ is reduced to the one of a square functional of type (\ref{eq:func1}), which was before proved.
\end{proof}

\section{Generalized Sobolev spaces}
For the definition, we refer the reader to the introduction.

\begin{prop} \label{prop:equivalence} For all $p\in (s_-,\infty)$ and $s\in(0,1)$, we have the following equivalence
$$ \|f\|_{L^p} + \| L^{s/m}(f)\|_{L^p} \simeq \|(1+L)^{s/m} f \|_{L^p}.$$
\end{prop}

\begin{proof} Set $\alpha=s/m$. We decompose $(1+L)^{\alpha}$ with the semigroup as following
\begin{align*}
 (1+L)^{\alpha} f & = \int_0^\infty e^{-t} e^{-tL} (1+L)(f) t^{1-\alpha} \frac{dt}{t} \\
 & = \int_0^\infty e^{-t} e^{-tL}(f) \frac{dt}{t^\alpha} + \int_0^\infty e^{-t} e^{-tL} (tL)^{1-\alpha}( L^\alpha f) \frac{dt}{t}.
\end{align*}
Since $e^{-tL}$  is uniformly bounded on $L^p$ (due to the off-diagonal decay), the $L^p$-norm of the first term is easily bounded by $\|f\|_{L^p}$. The second term is bounded by duality : indeed
\begin{align*}
 & \langle \int_0^\infty e^{-t} e^{-tL} (tL)^{1-\alpha}( L^\alpha f) \frac{dt}{t},g \rangle  =  \int_0^\infty e^{-t} \langle e^{-tL/2} (tL)^{\frac{1-\alpha}{2}}( L^\alpha f), e^{-tL^*/2} (tL^*)^{\frac{1-\alpha}{2}} g \rangle \frac{dt}{t}  \\
 & \hspace{1cm} \leq \int \left(\int_0^\infty \left|e^{-tL/2} (tL)^{\frac{1-\alpha}{2}}( L^\alpha f)\right|^2 \frac{dt}{t}\right)^{1/2} \left(\int_0^\infty \left|e^{-tL^*/2} (tL^*)^{\frac{1-\alpha}{2}}(g)\right|^2 \frac{dt}{t}\right)^{1/2} d\mu.
\end{align*}
Since $(1-\alpha)/2>0$, then the two square functionals are bounded in $L^p$ and $L^{p'}$ (by Proposition \ref{prop:Lp}) and that concludes the proof of
$$ \|(1+L)^{\alpha} f \|_{L^p} \lesssim \|f\|_{L^p} + \| L^{\alpha}(f)\|_{L^p}.$$
Let us now check the reverse inequality. As previously, for $u=0$ or $u=\alpha$ we write
$$ L^u f = \int_0^\infty e^{-t(1+L)} (1+L) L^u t^{1+\alpha} \frac{dt}{t} (1+L)^\alpha f.$$
By producing similar arguments as above, we conclude
$$ \|L^u(f)\|_{L^p} \lesssim \|(1+L)^\alpha f\|_{L^p},$$
which ends the proof.
\end{proof}

\begin{rem} The previous result legitimates the  designation ``{\it non-homogeneous Sobolev spaces}" for $W^{s,p}_L$, since its norm is equivalent to 
$$ \| \cdot \|_{W^{s,p}_L} \simeq \|(1+L)^{s/m} \cdot \|_{L^p}$$
which gives $W^{s,p}_L = (1+L)^{-s/m}(L^p)$.
\end{rem}

\subsection{Sobolev embeddings}

Here, we aim to prove Sobolev embeddings, with these new and general Sobolev spaces. To do that, we require an extra assumption : there exists a constant $c>0$ such that for all $x\in M$
\begin{equation} \mu(B(x,1)) \geq c. \label{eq:minoration} \end{equation}
Due to the homogeneous type of the manifold $M$, this is equivalent to a below control of the volume $(MV_d)$
\begin{equation} \tag{$MV_d$} \mu(B(x,r)) \gtrsim r^d \end{equation}
for all $0<r\leq 1$.

\begin{prop} Let $s>0$ be fixed and take $p\leq q$ such that
$$ \frac{1}{q}> \frac{1}{p}-\frac{s}{d} \quad \textrm{and} \quad p\geq
s_-.$$
Then under (\ref{eq:minoration}), we have the continuous embedding
$$ W^{s,p}_L \hookrightarrow L^q.$$
\end{prop}

\begin{proof} The desired embedding is equivalent to the following
inequality $$ \| f\|_{L^q} \lesssim \|(1+L)^{s/m} f\|_{L^p},$$
which is equivalent to
\begin{equation} \| (1+L)^{-s/m} f\|_{L^q} \lesssim \|f\|_{L^p}.
\label{eq:sobbis} \end{equation}
Let us prove this one. We first decompose the resolvent with the
semigroup as follows
$$ (1+L)^{-s/m} f = \int_0^\infty t^{s/m} e^{-t(1+L)}f \frac{dt}{t}.$$
Since we know that $e^{-tL}$ satisfies some $L^{p}-L^{q}$ off-diagonal
estimates (since $p>s_-$), it follows that it satisfies global estimates
$$ \|e^{-tL}\|_{L^p \rightarrow L^q} \lesssim
\min(1,t^{d(\frac{1}{q}-\frac{1}{p})/m}).$$
Indeed, from the off-diagonal decays, we know that for all balls $B$ of radius $t^{1/m}$, 
\begin{align*}
 \|e^{-tL}f\|_{L^q(B)} & \lesssim \mu(B)^{\frac{1}{q}} \sum_{j\geq 0} 2^{-j\delta} 
 \left(\aver{2^j B} |f|^{p} d\mu \right)^{1/p}. 
\end{align*}
So using the doubling property and Minkowski inequality, we have
\begin{align*}
 \|e^{-tL}f\|_{L^q} &  \simeq \left\| \left(\aver{(B(x,t^{1/m}))} |e^{-tL}(f)|^q d\mu \right)^{1/q}  \right\|_{L^q} \\
  & \lesssim \sum_{j\geq 0} 2^{-j\delta} \left\| \left(\aver{(B(x,2^j t^{1/m}))} |f|^p d\mu \right)^{1/p}  \right\|_{L^q} \\
  & \lesssim \sum_{j\geq 0} 2^{-j\delta}  \left( \int |f(y)|^p \mu(B(y,2^j t^{1/m}))^{p/q-1} d\mu(y) \right)^{1/p} \\
  & \lesssim \min(1,t^{d(\frac{1}{q}-\frac{1}{p})/m}) \|f\|_{L^p},
  \end{align*}
  where we used (\ref{eq:minoration}) with $q\geq p$ at the last equation.
Hence,
\begin{align*}
\| (1+L)^{-s/m} f\|_{L^q}  & \lesssim \left(\int_0^1 t^{s/m} e^{-t}
t^{d(\frac{1}{q}-\frac{1}{p})/m } \frac{dt}{t}+\int_1^\infty t^{s/m} e^{-t}
 \frac{dt}{t}\right) \|f\|_{L^p} \\
 & \lesssim \|f\|_{L^p},
 \end{align*}
since $s+d(\frac{1}{q}-\frac{1}{p})>0$. Finally, we have proved
(\ref{eq:sobbis}) which is equivalent to the desired result.
\end{proof}

In particular, we deduce :

\begin{cor} \label{cor:Li} Under the previous assumption, $ W^{s,p}_L \hookrightarrow L^\infty$ as soon as 
$$ s>\frac{d}{p} \textrm{ and } \quad p\geq s_-.$$
\end{cor}

\subsection{Limit Sobolev embedding into $L^\infty$}

We are interested to control the growth of the previous estimate with respect to the Sobolev norm, especially when $s$ tends to $d/p$. We refer the reader to \cite{KT} where a logarithmic Sobolev inequality by means of the BMO norm is proved. The goal is to reduce the behavior of the Sobolev norm in the previous Sobolev embeddings and to replace it by a BMO norm.
As described in \cite{KT}, this is crucial and very important to get a sharp estimate of the existence-time for solutions of Euler equations.

Moreover, such inequalities are interesting by themselves since they describe the rate of regularity to impose at a BMO function to prove its uniform boundedness. 

We propose a simpler proof than in \cite{KT} and extend it to our current framework. 

\begin{thm} Let $p\in(s_-,\infty)$ and $s>d/p$. We have the following Sobolev embedding~:
$$ \|f\|_{L^\infty} \lesssim 1+ \|f\|_{BMO} \left(1+\log(2+\|f\|_{W^{s,p}})\right)$$
as soon as $s>d/p$.
\end{thm}

\begin{proof}
Let us choose a small parameter $\epsilon<1$ and a large one $R>1$. We also have the following decomposition~:
\begin{equation} f = \phi(\epsilon L) f + \int_{\epsilon}^{R} \psi(tL) f \frac{dt}{t} + \zeta(R L) f  \label{eq:decompp} \end{equation}
where for a large enough integer $N>>s/m$, we define $\psi(z):=z^N e^{-z}$, $\phi(z)=\int_0^1 \psi(uz) \frac{du}{u}$ and $\zeta(z)=\int_1^\infty \psi(uz) \frac{du}{u}$.
Then, let us examine the three terms. \\
We first claim that
\begin{equation}
 \left\| \phi(\epsilon L) f \right\|_{L^\infty} \lesssim \epsilon^{s-d/p} \|f\|_{W^{s,p}}. \label{eq:sob1}
\end{equation}
Indeed, we have
\begin{align*}
 \phi(\epsilon L) f & = \int_0^1 \psi(\epsilon uL) f \frac{du}{u} = \int_0^\epsilon \psi(uL) f \frac{du}{u} \\
  & = \int_0^\epsilon u^{s/m} \tilde{\psi}(uL) L^{s/m} f \frac{du}{u},
\end{align*}
with $\tilde{\psi}(z)=z^{N-s/m} e^{-z}$. Then using the $L^{p}-L^\infty$ estimates of $\tilde{\psi}(uL)$ (implied by the off-diagonal decays), we conclude to (\ref{eq:sob1})~:
\begin{align*}
 \| \phi(\epsilon L) f\|_{L^\infty} & \lesssim \int_0^\epsilon u^{s/m} \left\|\tilde{\psi}(uL) L^{s/m} f\right\|_{L^\infty} \frac{du}{u} \\
 & \lesssim \int_0^\epsilon u^{s/m} u^{-d/(mp)} \left\| L^{s/m} f\right\|_{L^p} \frac{du}{u} \\
 & \lesssim \epsilon^{(s-d/p)/m} \|f\|_{W^{s,p}_L},
 \end{align*}
where we used $s>d/p$. Then, concerning the second term in (\ref{eq:decompp}), we claim
 \begin{equation}
 \left\|  \int_{\epsilon}^{R} \psi(tL) f \frac{dt}{t}  \right\|_{L^\infty} \lesssim \log(\epsilon^{-1} R) \|f\|_{BMO}. \label{eq:sob2}
\end{equation}
We introduce the averages (since $L({\bf 1})=0$) as follows and then use the off-diagonal decays~:
\begin{align*}
 \left\|  \int_{\epsilon}^{R} \psi(tL) f \frac{dt}{t}  \right\|_{L^\infty} & = \left\|  \int_{\epsilon}^{R} \psi(tL) [f-\aver{B(x,t^{1/m})}f] \frac{dt}{t}  \right\|_{L^\infty} \\
 & \leq \int_\epsilon ^R \left\| \psi(tL) [f-\aver{B(x,t^{1/m})}f] \right\|_{L^\infty} \frac{dt}{t} \\
 & \lesssim \sum_{k\geq 0} \int_\epsilon ^ R 2^{-k\delta} \left(\aver{B(x,2^k t^{1/m})} \left| f-\aver{B(x,t^{1/m})}f \right|^{s_-} d\mu \right)^{1/s_-} \frac{dt}{t}.
 \end{align*}
 It is well-known that
 \begin{align*}
 \left(\aver{B(x,2^k t^{1/m})} \left| f-\aver{B(x,t^{1/m})}f \right|^{s_-} d\mu \right)^{1/s_-} & \lesssim \left(\aver{B(x,2^k t^{1/m})}  \left| f-\aver{B(x,2^k t^{1/m})}f \right|^{s_-} d\mu \right)^{1/s_-} + k \|f\|_{BMO} \\
 & \lesssim (1+k) \|f\|_{BMO}.
 \end{align*}
 Therefore 
 \begin{align*}
 \left\|  \int_{\epsilon}^{R} \psi(tL) f \frac{dt}{t}  \right\|_{L^\infty}  & \lesssim \sum_{k\geq 0} \int_\epsilon ^ R 2^{-k\delta} (1+k) \|f\|_{BMO} \frac{dt}{t}\\
 & \lesssim \log(R/\epsilon) \|f\|_{BMO}
 \end{align*}
which yields (\ref{eq:sob2}).

For the third term in (\ref{eq:decompp}), we claim that 
\begin{equation}
 \left\|  \zeta(RL) f  \right\|_{L^\infty} \lesssim R^{-d/p} \|f\|_{BMO}. \label{eq:sob3}
\end{equation}
 We use similar arguments as we did for the first term~:
 \begin{align*}
   \| \zeta(R L) f\|_{L^\infty} & \lesssim \int_R^\infty \left\|\psi(uL)  f\right\|_{L^\infty} \frac{du}{u} \\
 & \lesssim \int_R^\infty  u^{-d/(mp)} \left\| f\right\|_{L^p} \frac{du}{u} \\
 & \lesssim R^{-d/(mp)} \|f\|_{L^p} \lesssim R^{-d/(mp)} \|f\|_{W^{s,p}}.
\end{align*}
Finally, we obtain the following estimate
$$ \|f\|_{L^\infty} \lesssim \left(\epsilon^{(s-d/p)/m}+R^{-d/(mp)}\right) \|f\|_{W^{s,p}} + \log(R/\epsilon) \|f\|_{BMO}.$$
We conclude as in \cite{KT} choosing $\epsilon$ and $R$ such that
$$ \epsilon^{(s-d/p)/m} = R^{-d/(mp)} = \min(1, \|f\|_{W^{s,p}}^{-1}).$$
\end{proof}

\mb We can slighty improve this result using the BMO space related with the semigroup $e^{-tL}$ as follows. Let us recall its definition~:

\begin{dfn} A function $f\in L^{s_-}_{loc}$ belongs to the space $BMO_L$ if for some exponent $p\in(s_-,\infty)$
$$ \|f\|_{BMO_L}:=\sup_{t>0,\ x\in M} \left(\aver{B(x,t^{1/m})} \left| f-e^{-tL}f\right|^{p} d\mu \right)^{1/p}.$$
\end{dfn}

We refer the reader to \cite{BZ2} (paragraph 3.4 for the special case given by a semigroup) for precise study of John-Nirenberg inequalities, showing that the norm in $BMO_L$ does not depend on the considered exponent $p$ as soon as $p\in(s_-,\infty)$. More generally, recent works of Jimenez del Toro, Martell and the two first authors \cite{BJM,BM} are devoted to the study of such self-improving inequalities.
Then it is well known that such $BMO_L$ spaces are bigger than the classical $BMO$ space (see Proposition 6.7 \cite{DY1} and Remark 7.6 of \cite{BZ} for a more general study of this question).

\begin{cor}  Let consider $p\in(s_-,\infty)$ and $s>0$. We have the following Sobolev embedding~:
$$ \|f\|_{L^\infty} \lesssim 1+ \|f\|_{BMO_L} \left(1+\log(2+\|f\|_{W^{s,p}})\right)$$
as soon as $s>d/p$.
\end{cor}

\begin{proof} We only mention the modifications and let the details to the reader. We keep the notations of the previous proof. Arguing as in the previous proof, we show that the first and third terms in (\ref{eq:decompp}) are still bounded. Concerning the second term, we claim that (instead of (\ref{eq:sob2})), we have
 \begin{equation}
 \left\|  \int_{\epsilon}^{R} \psi(tL) f \frac{dt}{t}  \right\|_{L^\infty} \lesssim \log(\epsilon^{-1} R) \|f\|_{BMO_L} + \left(\epsilon^{s-d/p}+R^{-d/p}\right) \|f\|_{W^{s,p}} . \label{eq:sob2bis}
\end{equation}
 This is based on the following identity
 $$ (1-2^{-N}) \int_{\epsilon}^{R} \psi(tL) f \frac{dt}{t} =  \int_{\epsilon/2}^{R} \psi(tL) (1-e^{-tL})f \frac{dt}{t} - \int_{\epsilon/2}^\epsilon \psi(tL) f \frac{dt}{t} + 2^{-N} \int_{R}^{2R} \psi(tL) f \frac{dt}{t},$$
this comes from $\psi(tL)e^{-tL}=2^{-N} \psi(2tL)$.
Then the second term can be bounded as for the first one in (\ref{eq:decompp}) and the third term as for the third one in (\ref{eq:decompp}). We also deduce that
$$ \left\|  \int_{\epsilon}^{R} \psi(tL) f \frac{dt}{t}  \right\|_{L^\infty} \lesssim \left\| \int_{\epsilon}^{2R} \psi(tL) (1-e^{-tL})f \frac{dt}{t} \right\|_{L^\infty} + \left(\epsilon^{(s-d/p)/m}+R^{-d/(mp)}\right) \|f\|_{W^{s,p}}. $$ 
Then it remains us to study the main term which is based (as previously) on
\begin{align*}  
\left(\aver{B(x,2^k t^{1/m})} \left|f-e^{-tL}f \right|^p d\mu \right)^{1/p} \lesssim \|f\|_{BMO_L},
 \end{align*}
since $B(x,2^k t^{1/m})$ admits a bounded covering of balls with radius $t^{1/m}$.
\end{proof}

\section{Proof of Theorem \ref{Sobalg1} using Paraproducts point of view} \label{sec:paraproduit}

In this section, we will prove Theorem \ref{Sobalg1} using a new tool in this topic which consists in paraproducts, associated to a semigroup (see \cite{B2} where they were recently introduced and in \cite{Phd}).

\begin{dfn} For $N$ a large enough integer, we set $\psi(x)=x^N e^{-x}(1-e^{-x})$, $ \phi(x):=-\int_x^\infty \psi(y) dy/y$ and
$$ \psi_t(L):=\psi(tL)=(tL)^N e^{-tL}(1-e^{-tL}) \quad \textrm{and} \quad \phi_t(L):=\phi(tL).$$
We also consider the two following kind of ``paraproducts''~:
$$ \Pi(f,g):= \int_{0}^\infty \psi(tL) \left[\phi(tL)f \, \phi(tL)g \right] \frac{dt}{t}$$
and
$$ \Pi_g(f):= \int_{0}^\infty \phi(tL) \left[\psi(tL)f \, \phi(tL)g \right] \frac{dt}{t}.$$
\end{dfn}

We get a ``spectral decomposition'' of the pointwise product as follows~: up to some numerical constant $c$, we have
$$ f  = c \int_0^\infty \phi'(tL) f \frac{dt}{t}.$$
So for two functions, we have
$$ fg := c^3 \int_{s,u,v>0} sL \phi'(sL) \left[sL \phi'(uL)f \, sL \phi'(vL)g \right] \frac{dsdudv}{suv}.$$
Since $\phi'(x)=\psi(x)/x$, one obtains (by splitting the integral into three parts according to $t:=\min\{s,u,v\}$)
\begin{align}
 fg := & c^3 \int_{0}^\infty \psi(tL) \left[\phi(tL)f \, \phi(tL)g \right] \frac{dt}{t} + c^3 \int_{0}^\infty \phi(tL) \left[\psi(tL)f \, \phi(tL)g \right] \frac{dt}{t} \nonumber\\
&+ c^3\int_{0}^\infty \phi(tL) \left[\phi(tL)f \, \psi(tL)g \right] \frac{dt}{t} \nonumber\\
& = c^3 \left[\Pi(f,g) + \Pi_g(f) + \Pi_f(g)\right] \label{paraproduct}. 
\end{align}

So the study of $fg$ is reduced to the study of the three paraproducts, appearing in this decomposition.

\begin{prop} \label{prop:para1}  Let $\beta>0$. For $p\in[r',\infty)$ and $q\in(s_-, \infty]$ with $p,r'\in(s_-,s_+)$ and $\frac{1}{r'}=\frac{1}{p}+\frac{1}{q}$
$$ \| L^{\beta}\Pi_g(f)\|_{L^{r'}} \lesssim \|L^\beta(f)\|_{L^p} \|g\|_{L^q}.$$
By symmetry, for $q\in[r',\infty)$ and $p\in(s_-, \infty]$ with $q,r'\in(s_-,s_+)$ and $\frac{1}{r'}=\frac{1}{p}+\frac{1}{q}$, we have
$$ \| L^{\beta}\Pi_f(g)\|_{L^{r'}} \lesssim \|f\|_{L^p} \|L^\beta (g)\|_{L^q}.$$
\end{prop}

Combining this result with H\"older inequality and using Proposition \ref{prop:equivalence}, we can also prove the following non-homogeneous version.

\begin{cor} \label{cor:para1i} 
Let $\alpha>0$ and set $\beta=\alpha/m>0$. For $p\in[r',\infty)$ and $q\in(s_-, \infty]$ with $p,r'\in(s_-,s_+)$ and $\frac{1}{r'}=\frac{1}{p}+\frac{1}{q}$
$$ \|\Pi_g(f)\|_{W^{\alpha,r'}_L} \lesssim \|f\|_{W_L^{\alpha,p}}  \|g\|_{L^q}.$$
By symmetry, for $q\in[r',\infty)$ and $p\in(s_-, \infty]$ with $q,r'\in(s_-,s_+)$ and $\frac{1}{r'}=\frac{1}{p}+\frac{1}{q}$, then
$$ \|\Pi_f(g)\|_{W_L^{\alpha,r'}} \lesssim \|f\|_{L^p} \|g\|_{W_L^{\alpha,q}}.$$
\end{cor}

\begin{proof}[Proof of Proposition \ref{prop:para1}]
We only show the homogeneous result (Proposition \ref{prop:para1}) and let the reader to check that the same argument still holds for the inhomogeneous framework. \\
Indeed, applying $L^\beta$ to $\Pi_g(f)$ yields
\begin{align*} 
L^\beta \Pi_g(f) & =  \int_{0}^\infty L^\beta \phi(tL) \left[\psi(tL)f \, \phi(tL)g \right] \frac{dt}{t} \\
 & =  \int_{0}^\infty \widetilde{\phi}(tL) \left[t^{-\beta} \psi(tL)f \, \phi(tL)g \right] \frac{dt}{t} \\
 & = \int_{0}^\infty \widetilde{\phi}(tL) \left[ \widetilde{\psi}(tL) L^\beta f \, \phi(tL)g \right] \frac{dt}{t},
\end{align*}
where we set $\widetilde{\phi}(z)=z^\beta \phi(z)$ and $\widetilde{\psi}(z)=z^{-\beta} \psi(z)$. So if the integer $N$ in $\phi$ and $\psi$ is taken sufficiently large, then $\widetilde{\phi}$ and $\widetilde{\psi}$ are still holomorphic functions with vanishing properties at $0$ and at infinity. As a consequence, we get
$$ L^\beta \Pi_g(f)  = \widetilde{\Pi}_g(L^\beta f)$$
with the new paraproduct $\widetilde{\Pi}$ built with $\widetilde{\phi}$ and $\widetilde{\psi}$.  Let us estimate this paraproduct.
By duality, for any smooth function $h\in L^r$ we have
\begin{align*}
  \langle L^\beta \Pi_g(f), h\rangle   & = \int \int_{0}^\infty \widetilde{\phi}(tL^*)h  \, \widetilde{\psi}(tL)(L^\beta f) \, \phi(tL)g \frac{dt}{t} d\mu \\
  & \leq \int \left(\int_0^\infty | \widetilde{\phi}(tL^*)h|^2 \frac{dt}{t} \right)^ {1/2} \left(\int_0^\infty |\widetilde{\psi}(tL)(L^\beta f)|^2 \frac{dt}{t} \right)^{1/2} \sup_t |\phi(tL)g| d\mu. 
\end{align*}
From the off-diagonal decay on the semigroup (\ref{eq:off-diag}), we know that
$$ \sup_t |\phi(tL)g(x)| \leq \M_{s_-}(g)(x)$$ 
and so by H\"older inequality
\begin{align*}
  \left| \langle L^\beta \Pi_g(f), h\rangle \right| \lesssim \left\| \left(\int_0^\infty |\widetilde{\phi}(tL^*)h|^2 \frac{dt}{t} \right)^ {1/2} \right\|_{L^{r}} \left\|\left(\int_0^\infty |\widetilde{\psi}(tL)(L^\beta f)|^2 \frac{dt}{t} \right)^{1/2}\right\|_{L^p}  \left\| {\M}_{s_-}g\right\|_{L^q}. 
\end{align*}
Since $\widetilde{\psi}$ and $\widetilde{\phi}$ are holomorphic functions vanishing at $0$ and having fast decays at infinity, we know from Proposition \ref{prop:Lp} that the two square functions are bounded on Lebesgue spaces. We also conclude the proof by duality, since it follows
\begin{align*}
  \left| \langle L^\beta \Pi_g(f), h\rangle \right| \lesssim \left\| h \right\|_{L^{r}} \left\|L^\beta f \right\|_{L^p}  \left\| g\right\|_{L^q}. 
\end{align*}
 \end{proof}

It remains to estimate the symmetric term $\Pi(f,g)$. For this term, the previous argument does not hold and we have to apply different arguments.

\begin{prop} \label{prop:para2} Let $\alpha\in(0,1)$, $\beta=\alpha/m\in(0,1/m)$ and assume Poincar\'e inequality $(P_s)$ for some $s<2$ and $ \delta>1+\frac{d}{s_-}$ (appearing in Assumption \ref{ass}). 
For $r'\in(1,\infty)$, $p_1\in[r',\infty)$, $q_1\in(r', \infty]$ and $q_2\in[r',\infty)$, $p_2\in(r', \infty]$ with $s\leq r'$, $r,r',p_1,q_2\in(s_-,s_+)$, $q_1,p_2\in(s_-,\infty]$ and
$$\frac{1}{r'}=\frac{1}{p_i}+\frac{1}{q_i}$$ we have
$$ \|L^\beta(\Pi(f,g))\|_{L^{r'}} \lesssim \|L^{\beta}(f)\|_{L^{p_1}} \|g\|_{L^{q_1}} + \|f\|_{L^{p_2}} \|L^{\beta}(g)\|_{L^{q_2}}.$$
\end{prop}

\begin{proof} First due to the self-improving property of Poincar\'e inequality (Theorem \ref{thm:kz}), we know that  without loss of generality, we can assume $s<r'$. Let us first recall the main quantity
$$ L^\beta \Pi(f,g):= \int_{0}^\infty L^\beta \psi(tL) \left[\phi(tL)f \, \phi(tL)g \right] \frac{dt}{t}.$$
Using the cancellation property $L^\beta({\bf 1})=0$, it follows that for all $x$
$$ L^\beta \Pi(f,g)(x):= \int_{0}^\infty L^\beta \psi(tL) \left[\phi(tL)f \, \phi(tL)g - \aver{B(x,t^{1/m})}(\phi(tL)f \, \phi(tL)g) \right](x) \frac{dt}{t}.$$
Fix $t>0$ and consider $h_t:=\phi(tL)f \, \phi(tL)g$. Using the off-diagonal decay of $(tL)^\beta \psi(tL)$, we deduce 
\begin{align*}
 \left| L^\beta \psi(tL) \left[h_t - \aver{B(x,t^{1/m})}h_t \right](x) \right| &
 \\
 &\hspace{-1.5cm} \lesssim  t^{-\beta} \sum_{j\geq 0} 2^{-j \delta} \left(\aver{B(x,2^j t^{1/m})} \left|h_t(y) - \aver{B(x,t^{1/m})}h_t \right|^{s_-} d\mu(y)\right)^{1/s_-}.
\end{align*}
And, using Poincar\'e inequality $(P_{s})$, which implies $(P_{\bar{s}})$ with $\bar{s}=\max(s,s_-)$, it follows that
\begin{align*}
 \left(\aver{B(x,2^j t^{1/m})}  \left|h_t - \aver{B(x,t^{1/m})}h_t \right|^{s_-} d\mu\right)^{1/s_-} & \leq  \left( \aver{B(x,2^j t^{1/m})}  \left|h_t - \aver{B(x,2^j t^{1/m})}h_t \right|^{\bar{s}} d\mu\right)^{1/\bar{s}} \\
 & \hspace{0.5cm} + \left(\aver{B(x,t^{1/m})}  \left|h_t - \aver{B(x,2^j t^{1/m})}h_t \right|^{\bar{s}} d\mu\right)^{1/\bar{s}} \\
 & \lesssim \left(\aver{B(x,2^j t^{1/m})}  \left|h_t - \aver{B(x,2^j t^{1/m})}h_t \right|^{\bar{s}} d\mu \right)^{1/\bar{s}} \\
 & \lesssim 2^{j(1+\frac{d}{s_-})} t^{1/m} \left(\aver{B(x,2^j t^{1/m})}  \left| \nabla h_t \right|^{\bar{s}} d\mu(y)\right)^{1/\bar{s}} \\
 & \lesssim 2^{j(1+\frac{d}{s_-})} t^{1/m}  {\M}_{\bar{s}}[\nabla h_t](x).
\end{align*}
Finally due to the doubling property and $\delta>1+\frac{d}{s_-}$, it comes
\begin{align*}
 \left| L^\beta \psi(tL) \left[h_t - \aver{B(x,t^{1/m})}h_t \right](x) \right| & \lesssim \sum_{j\geq 0} t^{\frac{1}{m}-\beta} 2^{-(\delta-1-\frac{d}{s_-}) j} {\M}_{\bar{s}}[\nabla h_t](x) \\
 & \lesssim t^{\frac{1}{m}-\beta} {\M}_{\bar{s}}[\nabla h_t](x).
\end{align*}
Hence, for all smooth function $h\in L^{r}$, we have (with $\widetilde{\psi}(z)=z^{\beta/2}\psi(z)^{1/2}$)
\begin{align*}
\left|\langle L^{\beta}\Pi(f,g), h \rangle \right| & \leq \int \int_{0}^\infty  \left|\widetilde{\psi}(tL)\left[t^{-\beta} h_t \right]  \widetilde{\psi}(tL^*)(h)\right| \frac{dt}{t}d\mu \\
 & \lesssim \int \left(\int_{0}^\infty \left|\widetilde{\psi}(tL) \left[t^{-\beta} h_t \right]\right|^2 \frac{dt}{t}\right)^ {1/2}  \left( \int_0^\infty \left|\widetilde{\psi}(tL^*)(h)\right|^2\frac{dt}{t}\right)^{1/2}d\mu \\
 & \lesssim \left\| \left(\int_{0}^\infty \left|\M_{\bar{s}} \left[t^{\frac{1}{m}-\beta} \nabla h_t \right]\right|^2 \frac{dt}{t}\right)^ {1/2}  \right\|_{L^{r'}} \|h\|_{L^r},
\end{align*}
where we use boundedness of the square function (Proposition \ref{prop:Lp}). 
Using Fefferman-Stein inequality for $\M_{\bar{s}}$ (with $\bar{s}=\max(s,s_-)<2,r'$) and duality, we obtain
$$ \left\|\langle L^{\beta}\Pi(f,g) \right\|_{L^{r'}}  \lesssim \left\| \left(\int_{0}^\infty \left|\left[t^{\frac{1}{m}-\beta} \nabla h_t \right]\right|^2 \frac{dt}{t}\right)^ {1/2}  \right\|_{L^{r'}}.$$
Since $\nabla h_t = \nabla \phi(tL)f \, \phi(tL)g + \phi(tL)f \, \nabla \phi(tL)g$, we get two terms. The operator $\phi(tL)$ is still bounded by the maximal function and consequently, we deduce
\begin{align*}
\| L^{\beta}(\Pi(f,g))\|_{L^{r'}}  & \lesssim \left\| \left(\int_{0}^\infty \left| t^{1/m-\beta} \nabla \phi(tL)f  \right|^2 \frac{dt}{t}\right)^ {1/2} \M_{s_-}(g) \right\|_{L^{r'}} \\
& + \left\| \left(\int_{0}^\infty \left| t^{1/m-\beta} \nabla \phi(tL)g \right|^2 \frac{dt}{t}\right)^ {1/2} \M_{s_-}(f) \right\|_{L^{r'}}.
\end{align*}
Using H\"older inequality, we finally get 
\begin{align*}
\| L^{\beta}(\Pi(f,g))\|_{L^{r'}}  & \lesssim \left\| \left(\int_{0}^\infty \left| t^{\frac{1}{m}-\beta} \nabla \phi(tL)f\right|^2 \frac{dt}{t}\right)^ {1/2} \right\|_{L^{p_1}}\|g\|_{L^{q_1}} \\
& + \left\| \left(\int_{0}^\infty \left| t^{\frac{1}{m}-\beta}\nabla \phi(tL)g\right|^2 \frac{dt}{t}\right)^ {1/2}\right\|_{L^{q_2}}\|f\|_{L^{p_2}}.
\end{align*}
Since the Riesz transform ${\mathcal R}:= \nabla L^{-1/m}$ is bounded on $L^p$ for $p\in (s_-,s_+)$, hence it satisfies $\ell^2$-valued inequality and
 \begin{align*}
\| L^{\beta}(\Pi(f,g))\|_{L^{r'}}  & \lesssim \left\| \left(\int_{0}^\infty \left| t^{\frac{1}{m}-\beta} {\mathcal R} L^{1/m} \phi(tL)f\right|^2 \frac{dt}{t}\right)^ {1/2} \right\|_{L^{p_1}}\|g\|_{L^{q_1}} \\
& \hspace{1cm} + \left\| \left(\int_{0}^\infty \left| t^{\frac{1}{m}-\beta}{\mathcal R} L^{1/m} \phi(tL)g\right|^2 \frac{dt}{t}\right)^ {1/2}\right\|_{L^{q_2}}\|f\|_{L^{p_2}} \\
& \lesssim \left\| \left(\int_{0}^\infty \left| t^{\frac{1}{m}-\beta} L^{1/m} \phi(tL)f\right|^2 \frac{dt}{t}\right)^ {1/2} \right\|_{L^{p_1}}\|g\|_{L^{q_1}} \\
& \hspace{1cm} + \left\| \left(\int_{0}^\infty \left| t^{\frac{1}{m}-\beta}  L^{1/m}\phi(tL)g\right|^2 \frac{dt}{t}\right)^ {1/2}\right\|_{L^{q_2}}\|f\|_{L^{p_2}}.
\end{align*}
With $\widehat{\psi}(z)=\phi(z) z^{1/m-\beta}$, we obtain
 \begin{align*}
\| L^{\beta}(\Pi(f,g))\|_{L^{r'}}  & \lesssim \left\| \left(\int_{0}^\infty \left| \widehat{\psi}(tL) L^\beta f\right|^2 \frac{dt}{t}\right)^ {1/2} \right\|_{L^{p_1}}\|g\|_{L^{q_1}} \\
& \hspace{1cm} + \left\| \left(\int_{0}^\infty \left| \widehat{\psi}(tL) L^\beta g\right|^2 \frac{dt}{t}\right)^ {1/2}\right\|_{L^{q_2}}\|f\|_{L^{p_2}}
\\
& \hspace{1cm} \leq \| L^\beta f\|_{L^{p_1}}\|g\|_{L^{q_1}}+\|L^{\beta}g\|_{L^{q_2}}\|f\|_{L^{p_2}}.
\end{align*}
We used that the square functions are bounded on Lebesgue spaces, since $\beta<1/m$, and $\widetilde{\psi}$ is holomorphic and vanishes at $0$ and at infinity, see Proposition \ref{prop:Lp}.  
\end{proof}

We can obtain a non-homogeneous version that we do not detail. Since we consider non-homogeneous regularity, the previous argument is necessary only for low scale ($t \leq 1$) so only a local Poincar\'e inequality is required.

\begin{cor} Let $\alpha\in(0,1)$, $\beta=\alpha/m$ and assume local Poincar\'e inequality $(P_{s,loc})$ and $\delta>1+\frac{d}{s_-}$. For $r'\in(1,\infty)$, $p_1\in[r',\infty)$, $q_1\in(r', \infty]$ and $q_2\in[r',\infty)$, $p_2\in(r', \infty]$ with $s\leq r'$, $r,r',p_1,q_2\in(s_-,s_+)$, $q_1,p_2\in(s_-,\infty]$ and
$$\frac{1}{r'}=\frac{1}{p_i}+\frac{1}{q_i}$$ we have
$$ \|\Pi(f,g)\|_{W_L^{\alpha,r'}} \lesssim \|f\|_{W_L^{\alpha,p_1}} \|g\|_{L^{q_1}} + \|f\|_{L^{p_2}} \|g\|_{W_L^{\alpha,q_2}}.$$
\end{cor}

Combining the decomposition \eqref{paraproduct} and Propositions \eqref{prop:para1} and \ref{prop:para2}, we get the following result.

\begin{thm}\label{thm:PSN} Assume (\ref{ass}) with $\delta>1+\frac{d}{s_-}$ and Poincar\'e inequality $(P_s)$.  Let $\alpha\in(0,1)$ and $r'>1$ with $s\leq r'<\infty$. For $p_1\in[r',\infty)$, $q_1\in(r', \infty]$ and $q_2\in[r',\infty)$, $p_2\in(r', \infty]$ verifying 
$$\frac{1}{r'}=\frac{1}{p_i}+\frac{1}{q_i}$$ and
$r,r',p_1,q_2\in(s_-,s_+)$, $q_1,p_2\in(s_-,\infty]$, we have
$$ \|L^{\alpha/m}(fg)\|_{L^{r'}} \lesssim \|L^{\alpha/m}(f)\|_{L^{p_1}} \|g\|_{L^{q_1}} + \|f\|_{L^{p_2}} \|L^{\alpha/m}(g)\|_{L^{q_2}}.$$
\end{thm}

\begin{proof}[Proof of Theorem \ref{Sobalg1}]
The proof follows now immediately from Theorem \ref{thm:PSN}.
\end{proof}

This point of view related to paraproducts is very suitable for studying the pointwise product of two functions. For more general nonlinearities, we would have to require a kind of ``paralinearization results" (as in the Euclidean case). This seems difficult and not really possible in such an abstract setting. However, we move the reader to a forthcoming work of Bernicot and Sire in this direction \cite{BeS}.

In order to get around this technical problem, we want to compare this approach with the one of \cite{CRT}, where the authors obtained characterizations of Sobolev norms involving square functionals (which are convenient to study the action of a nonlinearity). This is the aim of the following section.

\section{Characterization of Sobolev spaces via functionals}

As usual, we can expect to obtain a characterization of Sobolev norms by integrating the variations of the function. This will be the key tool for an alternative approach of Theorem  \ref{Sobalg2} and  \ref{Sobalg3}.

\begin{dfn} \label{def:S} Let $\rho>0$ be an exponent.
For a measurable function $f$ defined on $M$, $\alpha>0$ and $x\in M$, we define
$$ S_{\alpha}^\rho f(x)=\left(\int_{0}^{\infty} \left[\frac{1}{r^{\alpha}} \left(\frac{1}{\mu(B(x,r))}\int_{B(x,r)}|f(y)-f(x)|^\rho d\mu(y)\right)^{1/\rho}\right]^2\frac{dr}{r}\right)^{\frac{1}{2}}
$$
and
$$ S_{\alpha}^{\rho,loc} f(x)=\left(\int_{0}^{1} \left[\frac{1}{r^{\alpha}} \left(\frac{1}{\mu(B(x,r))}\int_{B(x,r)}|f(y)-f(x)|^\rho d\mu(y)\right)^{1/\rho}\right]^2\frac{dr}{r}\right)^{\frac{1}{2}}
$$
\end{dfn}
When $\rho=2$, these functionals are natural generalizations of those introduced by Strichartz in \cite{St}. \par
\noindent We will prove the following:

\begin{thm} \label{thm:caracterisation} Under Assumption (\ref{ass}) with $\delta$ sufficently large : $\delta>\alpha +d/s_-$  and  a local Poincar\'e inequality $(P_{s,loc})$ for some $s<2$, let $\alpha\in(0,1)$, $\beta=\alpha/m$. Then for all $p\in(s_-,s_+)$ with $s_-\leq \rho$ and $\max(\rho,s) <\min(2,p)$, there exist constants $c_1,c_2$ such that for all $f\in W^{\alpha,p}_L$  
$$ c_1 \left(\|L^{\beta}f\|_{L^p}+\|f\|_{L^p}\right) \leq  \|S_{\alpha}^{\rho,loc} f\|_{L^p}+\left\Vert f\right\Vert_{L^p} \leq c_2 \left(\|L^{\beta}f\|_{L^p}+\|f\|_{L^p}\right).$$
Moreover, if $M$ admits a global Poincar\'e inequality $(P_s)$, then  for all $f\in \dot W^{\alpha,p}_L$  
$$ c_1 \|L^{\beta}f\|_{L^p} \leq  \|S_{\alpha}^\rho f\|_{L^p} \leq c_2 \|L^{\beta}f\|_{L^p}.$$
\end{thm}

\begin{cor} As a consequence, under the above global assumptions, we obtain that
$$ \| \cdot \|_{\dot{W}^{\alpha,p}_L} \simeq \left\| S_{\alpha}^\rho (\cdot)\right\|_{L^p}$$
for all $p \in (\max(s,s_-),s_+)$.

In particular, the Sobolev space $W^{\alpha,p}_L$ depends neither on $L$ nor on $p\in (\max(s,s_-),s_+)$. 
\end{cor}

\begin{rem}[Self-improving property of the square functionals $S_{m\alpha}^\rho$]
This theorem shows that for $\beta\in(0,1/m)$, $p\in(s_-,s_+)$ and Poincar\'e inequality $(P_s)$ with $s\leq \min(2,p)$, we have
$$ \|L^{\alpha}f\|_{L^p} \simeq  \|S_{m\beta}^\rho f\|_{L^p} $$
as soon as
$$s_-\leq \rho <\min(2,p).$$
Since the map $\rho \rightarrow S_{m\beta}^\rho f(x)$ is non-decreasing, we deduce the following self-improving property~:

\mb
{\bf Property :} Under the above assumptions, if $S_{m\beta}^{s^-} \in L^p $ then $S_{m\beta}^{\rho} \in L^p$ for every $s_-\leq \rho <\min(2,p)$.
\end{rem}

The proof of Theorem \ref{thm:caracterisation} follows the ideas of \cite{CRT}. We  give the proof under global Poincar\'e inequality $(P_s)$. We refer to  \cite{CRT} for the local case, involving the local functional $S^{\rho,loc}_{\alpha}$.

\subsection{Proof of $\|S_{\alpha}^\rho f\|_{L^p} \lesssim \|L^{\beta}f\|_{L^p}$}
Assume that  $M$ admits the global Poincar\'e inequality $(P_s)$.
This paragraph is devoted to the proof of
$$ \|S_{\alpha}^\rho f\|_{L^p} \lesssim \|L^{\beta} f\|_{L^p} $$
requiring the assumption $\max(\rho,s) <\min(2,p)$.

\begin{proof} We first decompose the identity with the semigroup as  
\begin{align*}
 f& =-\int_0^{\infty}\frac{ \partial}{\partial{t}}(e^{-tL}f) dt \\
  & =\int_0^{\infty}(tL) e^{-tL}f \frac{dt}{t} \\
  & =\sum_{n=-\infty}^{\infty} \int_{2^n}^{2^{n+1} } (tL)e^{-tL}f \frac{dt}{t}.
\end{align*}
We set 
$$f_n:=\int_{2^n}^{2^{n+1} } (tL)e^{-tL}f \frac{dt}{t} $$
the piece at the scale $2^{n}$. Then, let us define
$$g_n:= \left(\int_{2^{n-1}}^{2^n } | (tL) e^{-tL}f)|^2  \frac{dt}{t^2}\right)^{1/2} $$
and
$$ h_n := \left(\int_{2^{n-1}}^{2^n}  |t^{1/m} \nabla (tL)e^{-tL} f)|^2 \frac{dt}{t^2}\right)^{1/2}.
$$ 
Observe that for all $x\in M$ and all integer $n$
\begin{equation}
 |f_n(x)| \leq {2^{n/2}} g_{n+1}(x). \label{eq:fngn}
\end{equation}
We claim that
\begin{equation}
|\nabla f_n| \lesssim 2^{n(1/2-1/m)} h_{n+1}. \label{eq:fnhn}
\end{equation}
Indeed, we have
$$
|\nabla f_n|\leq  \int_{2^n}^{2^{n+1} } \left|\nabla (tL)e^{-tL}f \right| \frac{dt}{t}$$
and then Cauchy Schwarz inequality with $t\simeq 2^n$ concludes also the proof of (\ref{eq:fnhn}). Using these elements, we will now estimate $S_{\alpha}f$:
\begin{align*}
S_{\alpha}^\rho f(x)^2 &= \int_0^{\infty} \left[\frac{1}{r^{\alpha}} \left(\frac{1}{\mu(B(x,r))} \int_{B(x,r)}|f(x)-f(y)|^\rho d\mu(y)\right)^{1/\rho} \right]^2\frac{dr}{r}
\\
& = \sum_{j=-\infty}^{+\infty}\int_{2^j}^{2^{j+1}} \left[\frac{1}{r^{\alpha}}\left(\frac{1}{\mu(B(x,r))} \int_{B(x,r)}|f(x)-f(y)|^\rho d\mu(y)\right)^{1/\rho} \right]^2\frac{dr}{r}
\\
&\lesssim \sum_{j=-\infty}^{+\infty}\left[\frac{1}{2^{j\alpha}} \left(\frac{1}{\mu(B(x,2^j))} \int_{B(x,2^{j+1})}|f(x)-f(y)|^\rho d\mu(y)\right)^{1/\rho} \right]^2.
\end{align*}
And the decomposition of $f$ by means of $f_n$ yields
\begin{align*}
\left(\aver{B(x,2^{j+1})} |f(x)-f(y)|^\rho d\mu(y)\right)^{1/\rho} & \leq  \sum_{n=-\infty}^{+\infty} \left(\aver{B(x,2^{j+1})} |f_n(x)-f_n(y)|^\rho d\mu(y)\right)^{1/\rho}.
\end{align*}
 Then for a fixed integer $n$, we split
  \begin{align*}
 \left(\aver{B(x,2^{j+1})} |f_n(x)-f_n(y)|^\rho d\mu(y)\right)^{1/\rho}
 &\lesssim  |f_n(x)-f_{n,B(x,2^{j+1})}| \\
&+ \left(\aver{B(x,2^{j+1})} |f_n(y)-f_{n,B(x,2^{j+1})}|^\rho d\mu(y)\right)^{1/\rho}
\\
&:=I+II.
 \end{align*}
Using doubling and Poincar\'e inequality $(P_s)$, we easily estimate $II$:
 \begin{align}
  II=\left(\aver{B(x,2^{j+1})} |f_n(y)-f_{n,B(x,2^{j+1})}|^\rho d\mu(y)\right)^{1/\rho}& \lesssim 2^{j+1}  {\mathcal M}_{\max(\rho,s)}(|\nabla f_n|)(x) 
  \\
  & \lesssim 2^{j} 2^{n(1/2-1/m)} \mathcal{M}_{\max(\rho,s)} h_{n+1}(x). \label{eq:fnhn2b}
\end{align}
It remains to estimate $I$. Take $B_0= B(x,2^{j+1})$. We construct for $i\geq 1$, the balls $B_{i} \subset B_0$ containing $x$ such that $B_{i}\subset B_{i-1}$ and $r(B_{i})=\frac{1}{2} r(B_{i-1})$. Since $f_{B_{i}}\underset{i\rightarrow\infty}{\longrightarrow}f(x)\;\mu-a.e.$,  using Poincar\'e inequality $(P_s)$, we get $\mu-a.e.$
\begin{align}
I=|f_n(x)-f_{n,B_0}|&\leq \sum_{i=1}^{\infty}|f_{n,B_{i}}-f_{n,B_{i-1}}| \nonumber
\\
&\lesssim \sum_{i=0}^{\infty}r(B_{i}) \left(\frac{1}{\mu(B_i)}\int_{B_i} |\nabla f_n|^s d\mu \right)^{1/s} \nonumber
\\
&\lesssim \sum_{i=0}^{\infty} r(B_{i}) {\mathcal M}_{s}(|\nabla f_n|)(x) \nonumber
\\
&\lesssim {\mathcal M}_{s}(|\nabla f_n|)(x)  \sum_{i=0}^{\infty} 2^{j+2-i} \nonumber
\\
&\lesssim 2^{j} 2^{n(1/2-1/m)} \mathcal{M}_s h_{n+1}(x).
 \label{eq:fnhn2}
\end{align}
Finally, using (\ref{eq:fngn}) for $j+1> (n-1)/m$ and (\ref{eq:fnhn2b}) with (\ref{eq:fnhn2}) for $j+1\leq (n-1)/m$, we obtain
\begin{align*}
 \left(\aver{B(x,2^{j+1})} |f(x)-f(y)|^\rho d\mu(y)\right)^{1/\rho} &\lesssim \sum_{n=-\infty}^{m(j+1)} 2^{n/2}{\mathcal M}_\rho (g_{n+1})(x)
 \\
 &+\sum_{n=m(j+1)+1}^{+\infty} 2^{j} 2^{n(1/2-1/m)} {\mathcal M}_{\max(\rho,s)} h_{n+1}(x).
\end{align*}
Let $c_n:=c_n(x):={\mathcal M}_\rho (g_{n})(x) + {\mathcal M}_{\max(\rho,s)} (h_{n})(x)$. Therefore 
\begin{align*}
 \sum_{j=-\infty}^{+\infty}\left[\frac{1}{2^{j\alpha}} \left(\aver{B(x,2^{j+1})}|f(x)-f(y)|^\rho d\mu(y)\right)^{1/\rho} \right]^2& \lesssim  \sum_{j=-\infty}^{+\infty} 2^{-2j\alpha} \left[\sum_{n=-\infty}^{m(j+1)} 2^{n/2}c_{n+1}\right]^2
 \\
 &+ \sum_{j=-\infty}^{+\infty}2^{-2j\alpha}\left[\sum_{n=m(j+1)+1}^{\infty} 2^j 2^{n(1/2-1/m)} c_{n+1}\right]^2.
\end{align*}
Let $$A:= \sum_{j=-\infty}^{+\infty}2^{-2j\alpha}\left[\sum_{n=-\infty}^{m(j+1)}2^{n/2}c_{n+1}\right]^2.$$
Choosing $\epsilon \in (0,\beta)$ and using Cauchy-Schwartz inequality yield
\begin{align*}
A &=  \sum_{j=-\infty}^{+\infty}2^{-2j\alpha}\left[\sum_{n=-\infty}^{m(j+1)}2^{n \epsilon}c_{n+1}2^{n/2}2^{-n\epsilon}\right]^2
\\ 
&\leq \sum_{j=-\infty}^{+\infty}2^{-2j\alpha}\left[\sum_{n=-\infty}^{m(j+1)}2^{2n\epsilon}\right]\left[\sum_{n=-\infty}^{m(j+1)} c_{n+1}^2 2^{n(1-2\epsilon)} \right]
\\
&\lesssim  \sum_{j=-\infty}^{+\infty}2^{2j(m\epsilon-\alpha)}  \sum_{n=-\infty}^{m(j+1)} c_{n+1}^2 2^{n(1-2\epsilon)}
\\
& \lesssim  \sum_{n=-\infty}^{+\infty} c_{n+1}^2 2^{n(1-2\epsilon)} \sum_{j=n/m-1}^{+\infty}2^{2j(m\epsilon-\alpha)}
\\
&\lesssim \sum_{n=-\infty}^{+\infty} c_{n+1}^2 2^{n(1-2\epsilon)}2^{2n(\epsilon - \beta)}
\\
&\lesssim \sum_{n=-\infty}^{+\infty} c_{n}^2 2^{n(1- 2\beta)}. 
\end{align*}
Let  $$B:= \sum_{j=-\infty}^{+\infty}2^{-2j\alpha}\left[\sum_{n=m(j+1)+1}^{\infty} 2^j 2^{n(1/2-1/m)} c_{n+1}\right]^2= \sum_{j=-\infty}^{+\infty} 2^{2j(1-\alpha)}\left[\sum_{n=m(j+1)+1}^{\infty} 2^{n(1/2-1/m)} c_{n+1}\right]^2.$$
Then, like  we did for $A$, considering $\epsilon \in (0, 1-\alpha)$, we obtain
\begin{align*}
B &=  \sum_{j=-\infty}^{+\infty}2^{2j(1-\alpha)}\left[\sum_{n=m(j+1)+1}^{\infty}2^{-n/m(1-\alpha-\epsilon)}c_{n+1}2^{n(1/2-1/m)} 2^{n/m(1-\alpha-\epsilon)}\right]^2
\\ 
&\lesssim \sum_{j=-\infty}^{+\infty}2^{2j\epsilon }\sum_{n=m(j+1)+1}^{\infty}c_{n+1}^2 2^{2n(1-\alpha-\epsilon)/m+n(1-2/m)}
\\
& \lesssim \sum_{n=-\infty}^{+\infty} 2^{n(1-2\alpha/m-2\epsilon/m)} c_{n+1}^2 \sum_{j=-\infty}^{(n-1)/m-1} 2^{2j\epsilon}
\\
&\lesssim \sum_{n=-\infty}^{+\infty} c_{n+1}^2 2^{n(1-2\beta)}.
\end{align*}
We deduce that
$$ S_{\alpha}^\rho f(x) \lesssim \left(\sum_{n=-\infty}^{+\infty} 2^{n(1-2\beta)} \left[{\mathcal M}_\rho (g_{n+1})+{\mathcal M}_{\max(\rho,s)} (h_{n+1}) \right]^{2}\right)^{1/2}.$$
Then for all $p>s$, using Fefferman-Stein inequality for ${\mathcal M}_s$ and ${\mathcal M}_\rho$ which are bounded on $L^p$ due to $s,\rho<p$ and $s,\rho <2$, we finally obtain 
\begin{align*}
\|S_{\alpha }^\rho f\|_{L^p} & \lesssim \left\| \left(\sum_{n=-\infty}^{+\infty} 2^{n(1-2\beta)} \left[{\mathcal M}_\rho g_n+{\mathcal M}_{\max(\rho,s)} h_n \right]^2 \right)^{1/2} \right\|_{L^p}
\\
&\lesssim \left\| \left(\int_{0}^{\infty } t^{1-2\beta} | (tL) e^{-tL}f)|^2 \frac{dt}{t^2}\right)^{1/2}\right\|_{L^p}+\left\| \left( \int_{0}^{\infty } t^{1-2\beta} |t^{1/m} \nabla (tL)e^{-tL} f)|^2 \frac{dt}{t^2}\right)^{1/2} \right\|_{L^p} \\
& \lesssim  \left\| \left(\int_{0}^{\infty } | (tL)^{1-\beta} e^{-tL} L^{\beta}f)|^2 \frac{dt}{t}\right)^{1/2}\right\|_{L^p}+\left\| \left( \int_{0}^{\infty } |t^{1/m} \nabla (tL)^{1-\beta} e^{-tL} L^\beta f)|^2 \frac{dt}{t^2}\right)^{1/2} \right\|_{L^p}.
\end{align*}
Since the two quadratic functionals are bounded (see Proposition \ref{prop:Lp}), we also conclude the proof of
$$ \|S_{\alpha }^\rho f\|_{L^p} \lesssim \|L^\beta(f)\|_{L^p}.$$
\end{proof}

\subsection{Proof of $\|L^{\beta}f\|_{L^p}\lesssim \|S_{\alpha}^\rho f\|_{L^p} $ for $p\in(s_-,s_+)$}
 
 This paragraph is devoted to the proof of
$$ \|L^{\beta} f\|_{L^p}\lesssim \|S_{\alpha}^\rho f\|_{L^p} $$
requiring the assumption $s_- \leq \rho$.
 
 \begin{proof} We begin by noting that from Proposition \ref{prop:Lp} with duality, we have  
for all $p\in(s-_,\infty)$,
 $$\|L^{\beta }f\|_{L^p}\leq  C_{\beta,p} \left\| \left(\int_0^{\infty}|(tL)^{1-\beta}e^{-tL}L^{\beta}f|^2\frac{dt}{t}\right)^{1/2}\right\|_{L^p}.$$
 So it suffices to prove that pointwise
 $$ \left(\int_{0}^{\infty}t^{1-2\beta}|Le^{-tL}f(x)|^2dt \right)^{1/2} \lesssim S_{\alpha}^\rho f(x).$$
 The $L^2$ analyticity of  the semigroup (see subsection \ref{subsec:semigroup}), the first point of assumption 1.8 and $e^{-tL}(1)=1$, yield with $B=B(x,t^{1/m})$
 \begin{align*}
 |Le^{-tL}f(x)|&= \left| t^{-1} (tL){\mathcal A}_t(f-f(x))(x)\right|
 \\
 &\lesssim t^{-1} \sum_{k\geq 0} 2^{- k \delta} \left(\aver{2^k B} |f(y)-f(x)|^{s_-} d\mu(y) \right)^{1/s_-} \\
& := \sum_{k\geq 0} I_t(k).
\end{align*}
Using this estimate and Minkowski inequality, one  obtains
\begin{align*}
\left(\int_0^\infty t^{1-2\beta} |\sum_{k\geq 0} I_t(k) |^2 dt\right)^{1/2} &
 \\
 &\hspace{-3cm}  \lesssim \sum_{k=1}^{\infty} 2^{-k\delta} \left(\int_0^{\infty} \frac{t^{-1-2\beta}}{\mu(B(x,t^{1/m}))^{2/s_-}}\left(\int_{B(x,2^{k+1}t^{1/m})} |f(x)-f(y)|^{s_-} d\mu(y)\right)^{2/s_-} dt\right)^{1/2}
\\
&\hspace{-3cm} \lesssim \sum_{k=0}^{\infty} 2^{-k\delta} \left(\int_0^\infty 2^{2k\alpha} \frac{r^{-1-2\alpha}}{\mu(B(x, 2^{-k-1}r)^{2/s_-}}\left(\int_{B(x,r)} |f(x)-f(y)|^{s_-} d\mu(y)\right)^{2/s_-} dr \right)^{1/2}
\\
&\hspace{-3cm}\lesssim \sum_{k=0}^{\infty} 2^{-k\delta} 2^{k\alpha}2^{kd/s_-} \left(\int_0^{\infty}\frac{r^{-1-2\alpha}}{\mu(B(x,r))^{2/s_-}}\left(\int_{B(x,r)} |f(x)-f(y)|^{s_-} d\mu(y)\right)^{2/s_-} dr\right)^{1/2}
\\
&\hspace{-3cm}\lesssim \left(\sum_{k=1}^{\infty} 2^{-k\delta} 2^{k\alpha} 2^{kd/s_-}\right) S_{\alpha}^{s_-} f(x) \\
&\hspace{-3cm}\lesssim S_{\alpha}^{s_-} f(x) \lesssim S_{\alpha}^{\rho} f(x)
\end{align*}
where in the last inequality, we use that the sum is finite since $ \delta>\alpha +d/s_-$ and then $s_-\leq \rho$.
\end{proof}

Now we are able to give an alternative proof of Theorem \ref{Sobalg2}:
\subsection{Second proof of  Theorem \ref{Sobalg2}}

\begin{rem} As remarqued in the introduction, Theorem \ref{Sobalg2} is a direct consequence of Theorem \ref{Sobalg1}, which was already proved in Section \ref{sec:paraproduit}, using the `paraproducts" point of view. Here, we obtain another proof via the previous characterization, involving the functionals $S^\rho_\alpha$. 
\end{rem} 

We will give the proof in the homogeneous case. The proof of the non-homogeneous case is analogous using $S_{\alpha}^{\rho,loc}$ instead of $S_\alpha^\rho$.
Let $f,\,g \, \in \dot W_L^{\alpha,p}$. Let $x\in M$. We take $\rho=s_-$ here. Then
\begin{align*}
S_{\alpha}^\rho (fg)(x)&=\left(\int_{0}^{\infty} \left[\frac{1}{r^{\alpha}} \left(\frac{1}{\mu(B(x,r))}\int_{B(x,r)}|f(y)g(y)-f(x)g(x)|^\rho d\mu(y)\right)^{1/\rho}\right]^2\frac{dr}{r}\right)^{\frac{1}{2}}
\\
&\leq\left(\int_{0}^{\infty} \left[\frac{1}{r^{\alpha}} \left(\frac{1}{\mu(B(x,r))}\int_{B(x,r)}|(f(y)-f(x))g(y)|^\rho d\mu(y)\right)^{1/\rho}\right]^2\frac{dr}{r}\right)^{\frac{1}{2}}
\\
&+\left(\int_{0}^{\infty} \left[\frac{1}{r^{\alpha}} \left(\frac{1}{\mu(B(x,r))}\int_{B(x,r)}|f(x)(g(y)-g(x))|^\rho d\mu(y)\right)^{1/\rho}\right]^2\frac{dr}{r}\right)^{\frac{1}{2}}
\\
&\leq \|g\|_{L^\infty}S_{\alpha}^\rho (f)(x)+ \|f\|_{L^\infty}S_{\alpha}^\rho (g)(x).
\end{align*}
Now using the second assertion of Theorem \ref{thm:caracterisation}, we deduce that
\begin{align*}
\|L^{\alpha/m}(fg)\|_{L^p} &\lesssim\|S_{\alpha}^\rho (fg)\|_{L ^p}
\\
&\lesssim \|g\|_{L^\infty}\|S_{\alpha}^\rho (f)\|_{L^p}+ \|f\|_{L^\infty}\|S_{\alpha}^\rho (g)\|_{L^p} 
\\
&\lesssim \|L^{\alpha/m}(f)\|_{L^p} \|g\|_{L^\infty}+\|L^{\alpha/m}(g)\|_{L^p}\|f\|_{L^\infty}
\end{align*}
which ends the proof of Theorem \ref{Sobalg2}.

\begin{proof}[Proof of Theorem \ref{Sobalg3}]
The proof of this theorem follows immediatly from Theorem \ref{Sobalg2} and the Sobolev embedding
$$
W^{\alpha,p}_L\subset L^{\infty}$$
when $\alpha p>d$ (see Corollary \ref{cor:Li}).
\end{proof}

\section{Higher order Sobolev spaces and nonlinearities preserving Sobolev spaces }

\subsection{Chain rule and Higher-order Sobolev spaces with a sub-Riemannian structure} \label{subsec:sub}

Since previous results cannot be extended (in such a general context) for higher order Sobolev spaces, we present how it is possible to do it in the context of a sub-Riemannian structure. Indeed such properties allow us to get a ``chain rule".

\subsubsection{Sub-Riemannian structure}

We assume that there exists $X:=\{X_k\}_{k=1,...,\K}$ a finite family of real-valued vector fields (so $X_k$ is defined on $M$ and $X_k(x)\in TM_x$) such that
\be{L} L=-\sum_{k=1}^\K X_k^2.\ee
 We identify the ${X_k}$'s with the first order
differential operators acting on Lipschitz functions  defined on
$M$ by the formula
$$  X_kf(x)=X_k(x)\cdot \nabla f(x), $$
and we set $Xf=(X_1 f, X_2f,\cdots, X_\K f)$ and
$$  |Xf(x)|=\left(\sum_{k=1}^\K |X_k f(x)|^2\right)^{1/2}, \quad
x\in M.$$
We define also the higher-order differential operators as follows : for $I\subset \{1,...,\K\}^k$, we set
$$  X_I := \prod_{i\in I}  X_{i}.$$
We assume the following and extra hypothesis:
\begin{ass} \label{ass:}
For every subset $I$, the $I$th-local Riesz transform ${\mathcal R}_I:=X_I (1+L)^{-|I|/2}$ and its adjoint ${\mathcal R}_I^*:=(1+L)^{-|I|/2} X_I $ are bounded on $L^p$ for every $p\in(s_-,s_+)$ (which is the range of boundedness for the Riesz transform $\nabla L^{-1/2}$).
\end{ass}

\subsubsection{Chain rule}

We refer the reader to \cite{CRT} and to \cite{BeS} for precise proofs of these results. 

\begin{lem} For every integer $k\geq 1$ and $p\in(s_-,s_+)$, 
$$ \|f\|_{W^{k,p}} \simeq \sum_{I\subset \{1,...,\K\}^k} \|X_I(f)\|_{L^p}.$$
\end{lem}

As a consequence of Theorem \ref{thm:caracterisation}

\begin{prop}[Proposition 19 \cite{CRT}] Let $\alpha:=k+t>1$ (with $k$ an integer and $t \in(0,1)$) and $p\in(s_-,s_+)$, then
\begin{equation} \label{recursive}
 f\in W^{\alpha,p}  \Longleftrightarrow f\in L^p \textrm{  and  } \forall I \subset \{1,...,\K \}^k,\  X_I(f) \in W^{t,p} 
 \end{equation}
 Moreover, under Assumption (\ref{ass}) with $\delta$ sufficently large : $\delta>t +d/s_-$  and  a local Poincar\'e inequality $(P_{s,loc})$ for some $s<2$. Then for all $p\in(s_-,s_+)$ with $s_-\leq \rho$ and $\max(\rho,s) <\min(2,p)$,
 \begin{equation}
 f\in W^{\alpha,p}  \Longleftrightarrow f\in L^p \textrm{ and } \forall I \subset \{1,...,\K \}^k, \  S_{t}^\rho(X_I(f)) \in L^p.  \label{eq:car}
\end{equation}
\end{prop}

\begin{rem} Note that the formulation of \eqref{recursive} is slightly different from the one of Proposition 19 in \cite{CRT}.
\end{rem}
We finish this section by describing some situations where this sub-Riemannian structure appears.

\begin{ex}
\begin{itemize}
\item Laplacian operators on Carnot-Caratheodory spaces. Let $\Omega$ be an open connected subset of ${\mathbb R}^d$ and
$Y=\{Y_k\}_{k=1}^\kappa$ a family of real-valued, infinitely
differentiable vector fields satisfying the usual H\"ormander condition (which means that the Lie algebra generated by the $X_i$'s is ${\mathbb R}^d$). Then we can define a Riemannian structure associated with these vector fields.

\item Lie groups. Let $M=G$ be a unimodular connected Lie group endowed with its Haar measure $d\mu=dx$ and assume that it has polynomial volume growth. Recall that ``unimodular" means that $dx$ is both left-invariant and right-invariant. Denote by ${\mathcal L}$ the Lie algebra of $G$. Consider a family $X =
\{X\, ... ,X_\kappa\}$ of left-invariant vector fields on $G$ satisfying the H\"ormander condition, which
means that the Lie algebra generated by the $X_i$'s is ${\mathcal L}$. We can build the Carnot-Caratheodory metric which brings a Riemannian structure on the group. In this situation, we know from \cite{NSW} that the group satisfies a local doubling property : $(D)$ is satisfied for $r \leq 1$. Then, the ``local'' results involving the non-homogeneous Sobolev spaces can be applied. Concerning the doubling property for large balls, two cases may occur: either the manifold is doubling or the volume of the balls admit an exponential growth \cite{Gui}. For example, nilpotents Lie groups satisfies the doubling property (\cite{16}). Particular case of nilpotent groups are Carnot groups, where the vector fields are given by a Jacobian basis of its Lie algebra and satisfy H\"ormander condition.\\
Considering the sub-Laplacian $L=- \sum X_i^2$, this frawework was already treated in \cite[Thm5.14]{Robinson} and \cite[Section 3, Appendix 1]{CRT}, in particular the heat semigroup $e^{-tL}$ satisfies Gaussian upper-bounds and Assumption \ref{ass:} on the higher-order Riesz transforms is satisfied too.

\item Particular cases of nilpotents Lie groups are the Carnot groups (if it admits a stratification), as for example the different Heisenberg groups. We refer the reader to \cite{GS} for an introduction of pseudodifferential operators in this context using a kind of Fourier transforms involving irreducible representations. 
\end{itemize}
\end{ex}

\subsection{Nonlinearities preserving Sobolev spaces}

\begin{prop} \label{prop:nonlinearity} Assume a local Poincar\'e inequality $(P_{s,loc})$ for some $s<2$. Let $F$ be a Lipschitz function on $\R$ then it continuously acts on some Sobolev spaces. More precisely, let $\alpha \in (0,1)$, $p\in(s_-,s_+)$ with $p\geq s$ and assume that $\delta$ (in Assumption \ref{eq:off-diag}) is sufficently large : $\delta>\alpha +d/s_-$. Then
\begin{itemize}
\item if $F$ is Lipschitz, we have
$$ \|F(f)\|_{W^{\alpha,p}_L} \lesssim \|f\|_{W^{\alpha,p}_L}.$$
\item if $F$ is locally Lipschitz then for every $R$ there exists a constant $c_R$ such that  for every $f\in W^{\alpha,p}_L \cap L^\infty$ with $\|f\|_{L^\infty}\leq R$ we have
$$  \|F(f)\|_{W^{\alpha,p}_L} \leq c_R \|f\|_{W^{\alpha,p}_L}.$$
\end{itemize}
\end{prop}

\begin{proof} The first claim is a direct consequence of Theorem \ref{thm:caracterisation} with $\rho=s_-$.
The second claim, comes from the following observation: since $F$ is locally Lipschitz, then the restricted function $\tilde{F}:=F_{|B(0,R)}$ is Lipschitz on $B(0,R)$. As for every function $f \in W^{\alpha,p}_L \cap L^\infty$ with $\|f\|_{L^\infty}\leq R$, we have 
$$ F(f) = \tilde{F}(f),$$
the first claim applied to $\tilde{F}$ ends the proof.
\end{proof}

Using the results of the previous subsection, it is possible to extend such results for higher-order Sobolev spaces in the context of a sub-Riemannian structure:

\begin{prop} \label{prop:nonlinearity2} Assume a local Poincar\'e inequality $(P_{s,loc})$ for some $s<2$. Assume that we are in the more-constraining context of a sub-Riemannian structure (as described in Subsection \ref{subsec:sub}) and consider $F$ a function on $\R$. Assume that $F\in W^{N,\infty}_{loc}$ for some integer $N\geq 1$ then it continuously acts on some Sobolev spaces. More precisely, let $\alpha \in (0,N)$, $p\in(s_-,s_+)$ with $p\geq s$. Then there exists $\delta_0>0$ such that, fpr all $\delta>\delta_0$, ($\delta$ occurs in Assumption \ref{eq:off-diag}),
\begin{itemize}
\item if $F\in W^{N,\infty}$, we have
$$ \|F(f)\|_{W^{\alpha,p}_L} \lesssim \|f\|_{W^{\alpha,p}_L}.$$
\item if $F\in W^{N,\infty}_{loc}$ then for every $R$ there exists a constant $c_R$ such that  for every $f\in W^{\alpha,p}_L \cap L^\infty$ with $\|f\|_{L^\infty}\leq R$ we have
$$  \|F(f)\|_{W^{\alpha,p}_L} \leq c_R \|f\|_{W^{\alpha,p}_L}.$$
\end{itemize}
\end{prop}

The proof  can be made by iterative arguments on $N$ as for Theorem 22 in \cite{CRT}. Else in \cite{BeS} a direct proof is detailed, the key observation relying on (\ref{eq:car}) is the fact that computing $X_I(F(f))$, we deduce that to bound $X_I(F(f))$ in $W^{t,p}$ is reduced to estimate quantities as
$$ h:= \left[\prod_{\beta=1}^l X_{i_\beta}\right] (f) F^{(n)}(f)$$
where $i_\beta \subset I$, $n\leq k$ and $\sum |i_\beta| = |I|\leq k$. 

\section{Well-posedness results for semilinear PDEs with regular data}

This section is devoted to some applications of algebra properties for Sobolev spaces, concerning well-posedness results for quasi-linear dispersive equations (Schr\"odinger equations) and quasi-linear heat equations associated to the operator $L$.

More precisely, we are interested in the two following problems.

\mb {\bf Schr\"odinger equation~:}

Let $u_0\in L^2(M,{\mathbb C})$ and $F:{\mathbb C} \rightarrow {\mathbb C}$ a smooth function, we are interested in the equation
\begin{equation}
 \left\{ \begin{array}{l}
          i\partial_t u + L u = F(u)  \\
          u(0,.)=u_0.
         \end{array} \right. \label{eq:schr}
\end{equation}

\mb {\bf Heat equation~:}

Let $u_0\in L^p$ and $F:\R \rightarrow \R$ a smooth function, we are interested in the equation
\begin{equation}
 \left\{ \begin{array}{l}
          \partial_t u + L u = F(u)  \\
          u(0,.)=u_0.
         \end{array} \right. \label{eq:heat}
\end{equation}

\mb Here $F(u)$ is the nonlinearity. We refer the reader to \cite{Tao} for precise study of (\ref{eq:schr}) in the Euclidean setting with particular nonlinearities. In this section, we give applications of the previous results (concerning Sobolev spaces), proving well-posedness results for these equations in a general setting.

\subsection{Schr\"odinger equations on a Riemannian manifold}

Assume that $L$ satisfying the assumptions of Subsection (\ref{subsec:semigroup}) is a self-adjoint operator and that Poincar\'e inequality $(P_s)$ holds for some $s<2$.

 Then, spectral theory allows us to build the unitary semigroup $(e^{itL})_{t}$, bounded in $L^2$. Duhamel's formula formally yields
\begin{equation}
 u(t) = e^{itL}u_0 -i\int_0^t e^{i(t-\tau) L} F(u(\tau)) d\tau.
 \label{eq:duhamel}
\end{equation}

\mb We assume that $F$ is smooth in the following sense: identifying ${\mathbb C} ={\mathbb R}^2$, $F$ is smooth as a function from $\R^2$ to $\R^2$. Under this assumption, we have the following result~:

\begin{lem} \label{lem:Flip} The map $F$  acts continuously on the Sobolev spaces. More precisely, for $\alpha\in(0,1)$ and every $R$ (and uniformly with respect to $u,v\in W_L^{\alpha,2}$ with $\|u\|_{W^{\alpha,2}_L\cap L^\infty},\|v\|_{W^{\alpha,2}_L\cap L^\infty}\leq R$), we have
$$ \|F(u)-F(v)\|_{W^{\alpha,2}_L \cap L^\infty} \lesssim_R \|u-v\|_{W^{\alpha,2}_L\cap L^\infty}.$$
The assumption $\alpha\in(0,1)$ can be weakened to $\alpha>0$ in the context of a sub-Riemannian structure (as described in Subsection \ref{subsec:sub}).
\end{lem}

\begin{proof} Since $F$ is locally Lipschitz, the $L^\infty$-bound on $F(u)-F(v)$ is easily obtained. Let us focus now on the Sobolev norm. Since $F$ is assumed to be smooth, we can write for $x,y\in {\mathbb C}$
$$ F(y)-F(x) = (y-x)G(x,y) + (\overline{y-x}) H(x,y)$$
with
$$ G(x,y):=\int_0^1 \partial_z F(x+t(y-x)) dt \qquad H(x,y):=\int_0^1 \partial_{\overline{z}} F(x+t(y-x)) dt.$$
 Then, applying the algebra property of Sobolev spaces $W^{\alpha,2}_L \cap L^\infty$ (see Theorem \ref{Sobalg1}), it comes
$$ \|F(u)-F(v)\|_{W^{\alpha,2}_L} \lesssim \|u-v\|_{W^{\alpha,2}_L \cap L^\infty} \left( \|G(u,v)\|_{W^{\alpha,2}_L \cap L^\infty} + \|H(u,v)\|_{W^{\alpha,2}_L\cap L^\infty}\right).$$
In addition, the two functions $G$ and $H$ are smooth too and so locally Lipschitz. Using Proposition \ref{prop:nonlinearity}, we deduce that
$$ \|G(u,v)\|_{W^{\alpha,2}_L\cap L^{\infty}} \lesssim_R \|u\|_{W^{\alpha,2}_L\cap L^\infty}+\|v\|_{W^{\alpha,2}_L\cap L^\infty}$$
with an implicit constant, depending on the $L^\infty$-norm of $u,v$. The same holds for $H$ and the proof is therefore complete. In a sub-Riemannian context, we conclude using Proposition \ref{prop:nonlinearity2} instead of Proposition \ref{prop:nonlinearity}.
 \end{proof}

\begin{thm} Let $u_0\in W^{\alpha,2}_L$ with some $\alpha>d/2$ and $\alpha\in(0,1)$. Then $W^{\alpha,2}_L$ is continuously embedded in $L^\infty$ and
\begin{itemize}
 \item there exists a unique solution $u\in C^0_I W^{\alpha,2}_L$ of (\ref{eq:schr}) on some small enough time-interval $I$
 \item the solution $u$ depends continuously on the data.
\end{itemize}
The assumption $\alpha\in(0,1)$ can be weakened to $\alpha>0$ in the context of a sub-Riemannian structure (as described in Subsection \ref{subsec:sub}).
\end{thm}

\begin{proof} 
Let the time-interval $I$ be fixed and consider the map $D$ on $C^0_I W^{\alpha,2}_L$ defined by
 $$ D(f) = -i\int_0^t e^{i(t-s) L} f(s) ds.$$
We follow the reasoning of Section 3.3 in \cite{Tao} and adapt it to our current framework. So we first check that $D$ is bounded on $C^0_I W^{\alpha,2}_L$: indeed for all $t\in I$
\begin{align*} \|D(f)(t)\|_{W^{\alpha,2}_L} & = \| (1+L)^{\alpha/m} D(f)(t)\|_{L^2} \\
 & \leq  \int_0^t \| (1+L)^{\alpha/m} e^{i(t-\tau) L} f(\tau)  \|_{L^2} d\tau \\
 & \leq t \sup_{\tau\in I} \|(1+L)^{\alpha/m}f(\tau)\|_{L^2} \\
 & \leq |I| \|f\|_{C^0_I W^{\alpha,2}_L}.
\end{align*}
In addition, it is easy to check that the unitary semigroup preserves the Sobolev norms~: for all $t\in\R$
$$ \| e^{itL}u_0 \|_{W^{\alpha,2}_L} = \|u_0\|_{W^{\alpha,2}_L}.$$
Then, following Proposition 3.8 in \cite{Tao}, it remains  to check that $F$ is locally Lipschitz, which was done in Lemma \ref{lem:Flip}. As a consequence, we can apply the Duhamel's iteration argument (see Proposition 1.38 in \cite{Tao} and \cite{Seg} for introduction of such arguments): for all $u_0\in W^{\alpha,2}_L $, there exists a unique solution $u\in  C^0_I W^{\alpha,2}_L$ of 
$$ u = e^{itL}u_0 + DF(u)$$
on $I$ as soon as $|I|$ is small enough.
\end{proof}

\subsection{Heat equations on a Riemannian manifold}

Assume that $L$ verifies assumptions in Subsection \ref{subsec:semigroup} and $F$ is a smooth real function.

\begin{thm} Let $p\in(s_-,\infty)$ with $p\geq s$, $u_0\in W^{\alpha,p}_L\cap L^{\infty}$ with $\alpha<1$ and set $S^{\alpha,p}_L:=W^{\alpha,p}_L \cap L^\infty$.
\begin{itemize}
 \item there exists a unique solution $u\in C^0_I S^{\alpha,p}_L$ of (\ref{eq:heat}) on some small enough time-interval $I$
 \item the solution $u$ depends continuously on the data.
\end{itemize}
The assumption $\alpha\in(0,1)$ can be weakened to $\alpha>0$ in the context of a sub-Riemannian structure (as described in Subsection \ref{subsec:sub}).
\end{thm}

\begin{proof} We leave the details to the reader, since the proof exactly follows the previous one. Indeed, Duhamel's formula gives for $u$ a solution of (\ref{eq:heat})
 $$ u(t) = e^{-tL}u_0 -\int_0^t e^{-(t-\tau) L} F(u(\tau)) d\tau.$$
The same argument still holds considering the heat semigroup $(e^{-tL})_{t>0}$ instead of the unitary semigroup.
We have to check that the map
$$ D(f) = -\int_0^t e^{-(t-s) L} f(s) ds.$$
is bounded on $C^0_T S^{\alpha,p}_L$. The control of the Sobolev norm still holds and the $L^\infty$-norm is bounded since the off-diagonal decay (\ref{eq:off-diag}) implies the $L^\infty$-boundedness of the semigroup. Arguing similarly, we prove that $F$ is locally Lipschitz on $S^{\alpha,p}_L$. We also conclude as for the previous PDEs, by invoking Duhamel's iteration argument.
\end{proof}


\begin{thebibliography}{99}

\bibitem{ambrosio1}
L. Ambrosio, M. Miranda Jr and D.~Pallara, 
\newblock {\it Special functions of bounded variation in doubling metric measure spaces}, 
Calculus of variations~: topics from the mathematical heritage of E. De Giorgi, Quad. Mat., Dept. Math,
  Seconda Univ. Napoli, Caserta \textbf{14} (2004), 1--45.

\bibitem{Asterisque}
P. Auscher and P. Tchamitchian,
\newblock Square Root Problem for Divergence Operators and Related Topics, 
\newblock {\it Ast{\'e}risque} \textbf{249} (1998), Soc. Math. France

\bibitem{A}
P.~Auscher,
\newblock On necessary and sufficient conditions for ${L}^p$ estimates of {R}iesz transforms associated to elliptic operators on ${\R}^n$ and related estimates,
\newblock {\it Memoirs of Amer. Math. Soc.} \textbf{186} no. 871 (2007).

\bibitem{AC}
P.~Auscher and T.~Coulhon.
\newblock Riesz transform on manifolds and Poincar\'e inequalities.
\newblock {\it Ann. Sc. Nor. Sup. Pisa} \textbf{5} (2005), IV, 531--555.

\bibitem{AMT}
P. Auscher, A. McIntosh and P. Tchamitchian,
\newblock Noyau de la chaleur d'op\'rateurs elliptiques complexes,
\newblock {\em Math. Research Lett.} \textbf{1} (1994), 37--45.

\bibitem{AMT2}
P. Auscher, A. McIntosh and P. Tchamitchian,
\newblock Heat kernel of complex elliptic operaterators and applications,
\newblock {\em Journ. Funct. Anal.} \textbf{152} (1998), 22--73.

\bibitem{amr} 
P. Auscher, A. McIntosh and E. Russ, 
\newblock Hardy spaces of differential forms on Riemannian manifolds, 
\newblock {\it J. Geom. Anal.} {\bf 18}, 1 (2008), 192--248. 

\bibitem{BJM}
 N. Badr, A. Jim\'enez-del-Toro and J.M. Martell,
\newblock $L^p$ self-improvement of generalized Poincar\'e inequalities in spaces of homogeneous type,
\newblock  {\it J. Funct. Anal.} \textbf{260} (2011), no. 11, 3147--3188.

\bibitem{BZ}
F. Bernicot and J. Zhao, 
\newblock New Abstract Hardy Spaces,
\newblock {\it J. Funct. Anal.} \textbf{255} (2008), 1761-1796.

\bibitem{BZ2}
F. Bernicot and J. Zhao,
\newblock Abstract framework for John Nirenberg inequalities and applications to Hardy spaces,
\newblock {\it Ann. Sc. Nor. Sup. Pisa} to appear in 2012.

\bibitem{B2}
 F. Bernicot, A $T(1)$-Theorem in relation to a semigroup of operators and applications to new paraproducts,
  {\it Trans. of Amer. Math. Soc.} (2012), available at http://arxiv.org/abs/1005.5140.

\bibitem{BM}
F. Bernicot and C. Martell, 
\newblock Self-improving properties for abstract Poincar\'e type inequalities, 
\newblock available on arxiv

\bibitem{BeS}
F. Bernicot and Y. Sire, 
\newblock Propagation of low regularity for solutions of nonlinear PDEs on a Riemannian manifold with a sub-Laplacian structure, 
\newblock available on arxiv

\bibitem{Bourdaud}
G. Bourdaud, 
\newblock Le calcul fonctionnel dans les espaces de Sobolev,
\newblock  {\it Invent. Math.} \textbf{104} (1991), 435--446. 

\bibitem{BS}
G.  Bourdaud and W. Sickel,  
\newblock Composition operators on function spaces with fractional order of smoothness,  
 \newblock {\it to appear} RIMS Kokyuroku Bessatsu.

%\bibitem{cm} 
%R.R. Coifman and Y. Meyer, Au-del\`a des op\'erateurs pseudo-diffe\'rentiels,
%Ast\'erisque \textbf{57}, Soci\'et\'e Math. de France, 1978.

\bibitem{coifman2}
R. Coifman and G. Weiss, 
\newblock {\it Analyse harmonique sur certains espaces homog\`{e}nes}, 
\newblock Lecture notes in Math. \textbf{242} (1971).

\bibitem{CD}
T.~Coulhon and X.T. Duong.
\newblock \emph{Riesz transforms for $1\leq p\leq 2$.}
\newblock  Trans. Amer. Math. Soc. 351, (2), 1151--1169,
  1999.

\bibitem{CRT}
T. Coulhon, E. Russ and V. Tardivel-Nachef,
\newblock {\it Sobolev algebras on Lie groups and Riemannian manifolds}, 
\newblock Amer. J. of Math. \textbf{123} (2001), 283--342.

\bibitem{16}
N. Dungey, A.F.M. ter Elst and D. W. Robinson,
\newblock Analysis on Lie groups with polynomial growth,
\newblock {\it Progress in Mathematics} \textbf{214}, 2003, Birk\"auser Boston Inc., Boston, MA.

\bibitem{DY1}
X.T. Duong and L. Yan,
\newblock Duality of {H}ardy and {BMO} spaces associated with operators with heat kernel bounds,
\newblock {\it J. Amer. Math. Soc.} \textbf{18}, no.4 (2005), 943--973.

\bibitem{DY}
X.T. Duong and L. Yan,
\newblock New function spaces of {BMO} type, the {J}ohn-{N}iremberg inequality, {I}nterplation and {A}pplications,
\newblock {\it Comm. on Pures and Appl. Math.} \textbf{58} no.10 (2005), 1375--1420.

\bibitem{Phd}
D. Frey,
\newblock Paraproducts via $H^\infty$-functional calculus and a $T(1)$-Theorem for non-integral operators,
\newblock Phd Thesis, available at http://digbib.ubka.uni-karlsruhe.de/volltexte/documents/1687378 

\bibitem{GS}
I. Gallagher and Y. Sire,
\newblock Besov algebras on Lie groups of polynomial growth and related results,
\newblock available at http://arxiv.org/abs/1010.0154

\bibitem{grigo} A. Grigor'yan, The heat equation on a non-compact 
Riemannian manifold, {\it Math. USSR Sb.} {\textbf 72} (1) (1992), 47--77.

\bibitem{Gui}
Y. Guivarc'h, 
Croissance polynomiale et p\'eriodes des fonctions harmoniques, 
{\it Bull. Soc. Math. France} \textbf{101} (1973), 333--379.

\bibitem{GK}
A. Gulisashvili and M. A. Kon,
\newblock Exact smoothing properties of Schr\"odinger semigroups,  {\it Amer. J. Math.}, \textbf{118} (1996), 1215--1248.

\bibitem{hajlasz4}
P. Hajlasz and P. Koskela, 
\newblock {S}obolev met {P}oincar\'{e}, 
\newblock {\it Mem. Amer. Math. Soc.} \textbf{145} (2000), no. 688, 1--101.

\bibitem{hofmay} S. Hofmann, S. Mayboroda, Hardy and $BMO$ spaces associated to divergence form elliptic operators, {\it Math. Ann.} {\bf 344}, 37-116, 2009.

\bibitem{hmm} S. Hofmann, S. Mayboroda, A. McIntosh, Second order elliptic operators with complex bounded measurable coefficients in $L^p$, Sobolev and Hardy spaces, to appear in {\it Ann. Sci. E.N.S.}

\bibitem{KP}
T. Kato and G. Ponce,
\newblock Commutator estimates and the Euler and Navier-Stokes equations,
\newblock {\it Comm. Pure Appl. Math.} \textbf{41} (1988), 891--907.

\bibitem{KZ}
S. Keith and X. Zhong, 
\newblock The {P}oincar\'e inequality is an open ended condition, 
\newblock {\it Ann. of Math.} \textbf{167} (2008), no. 2, 575--599.

\bibitem{KT}
H. Kozono and Y. Taniuchi,
\newblock Limiting Case of the Sobolev Inequality in BMO, with Application to the Euler Equations,
\newblock {\it Commun. Math. Phys.} \textbf{214} (2000), 191--200.

\bibitem{Mc}
A. McInstosh,
\newblock Operators which have an $H_\infty$-calculus, 
\newblock {\it Miniconference on operator theory and partial differential equations} (1986) Proc. Centre Math. Analysis, ANU, Canberra \textrm{14}, 210--231.

\bibitem{NSW}
A. Nagel, E.M. Stein and S. Wainger,
\newblock Balls and metrics defined by vector fields I: Basic properties,
\newblock {Acta Math.} \textbf{155} (1985), 103--147.

\bibitem{Robinson}
D. W. Robinson, 
{\it Elliptic Operators and Lie Groups},
 Oxford Univerisity Press, 1991.

\bibitem{RS}
T. Runst and W. Sickel, 
\newblock Sobolev spaces of fractional order, Nemytskij operators, and nonlinear partial differential equations,
\newblock De Gruyter, Berlin, 1996.

\bibitem{saloff-coste} L. Saloff-Coste, Parabolic Harnack inequality
for divergence form second order differential operators, {\it Pot. 
Anal.} {\textbf 4}, 4 (1995), 429--467.

\bibitem{SC}
L. Saloff-Coste, 
\newblock {\it Aspects of Sobolev type inequalities}, 
\newblock Cambridge Univ., 2001.

\bibitem{Seg}
I.E. Segal,
\newblock Non-linear semi-groups,
\newblock {\it Ann. of Maths.} \textbf{78} (1963), 339--364.

\bibitem{Sickel}
W. Sickel, 
\newblock Necessary conditions on composition operators acting on Sobolev spaces of 
fractional order. The critical case $1 < s < n/p$. I, II and III.
\newblock  {\it Forum Math.} \textbf{9} (1997), 267--302. 
\newblock  {\it Forum Math.} \textbf{10} (1998), 199--231. 
\newblock {\it Forum Math. } \textbf{10} (1998), 303--327. 

\bibitem{St}
R. Strichartz, Multipliers on fractional Sobolev spaces, {\it  J. Math. Mech} \textbf{16 (9)}, (1967), 1031--1060.

\bibitem{Tao}
T. Tao,
\newblock Nonlinear dispersive equations, Local and global analysis,
\newblock {\it CBMS} \textbf{106}.

\end{thebibliography}
\end{document}